\documentclass[english, 11pt]{amsart}

\usepackage{xcolor}

\usepackage{bbm}
\numberwithin{equation}{section}

\usepackage{amsrefs}

\usepackage{fullpage}
\usepackage[T1]{fontenc}
\usepackage[utf8]{inputenc}
\usepackage{mathrsfs}
\usepackage{amsmath}
\usepackage{amsthm}
\usepackage{amssymb}

\theoremstyle{plain}
\newtheorem{thm}{\protect\theoremname}[section]
\theoremstyle{definition}
\newtheorem{condition}[thm]{\protect\conditionname}
\theoremstyle{remark}

\theoremstyle{plain}
\newtheorem{cor}[thm]{\protect\corollaryname}
\theoremstyle{plain}
\newtheorem*{thm*}{\protect\theoremname}
\theoremstyle{plain}
\newtheorem{prop}[thm]{\protect\propositionname}
\theoremstyle{plain}
\newtheorem{lem}[thm]{\protect\lemmaname}
\theoremstyle{plain}
\newtheorem*{lem*}{\protect\lemmaname}
\theoremstyle{definition}
\newtheorem{definition}[thm]{Definition}

\makeatother

\usepackage{babel}
\providecommand{\conditionname}{Condition}
\providecommand{\corollaryname}{Corollary}
\providecommand{\lemmaname}{Lemma}
\providecommand{\propositionname}{Proposition}
\providecommand{\remarkname}{Remark}
\providecommand{\theoremname}{Theorem}

\begin{document}
\global\long\def\iprod#1#2{\left\langle #1,\,#2\right\rangle }%

\global\long\def\EE{\mathbb{E}}%

\global\long\def\PP{\mathbb{P}}%

\global\long\def\i{\mathrm{i}}%

\global\long\def\d#1{\,{\rm d}#1}%

\global\long\def\R{\mathbb{R}}%

\global\long\def\C{\mathbb{C}}%

\global\long\def\Z{\mathbb{Z}}%

\global\long\def\N{\mathbb{N}}%

\global\long\def\Q{\mathbb{Q}}%

\global\long\def\T{\mathbb{T}}%

\global\long\def\F{\mathbb{F}}%

\global\long\def\vp{\varphi_n}%

\global\long\def\Sph{\mathbb{S}}%

\global\long\def\sub{\subseteq}%

\global\long\def\one{\mathbbm1}%

\global\long\def\vol#1{\text{vol}\left(#1\right)}%

\global\long\def\EE{\mathbb{E}}%

\global\long\def\sp{{\rm sp}}%

\global\long\def\iprod#1#2{\langle#1,\,#2\rangle}%

\global\long\def\conv#1{{\rm conv}\left(#1\right)}%

\global\long\def\supp#1{{\rm supp}\left(#1\right)}%

\global\long\def\eps{\varepsilon}%

\global\long\def\PP{\mathbb{P}}%

\global\long\def\ovr#1{{\rm ovr\left(#1\right)}}%

\global\long\def\norm#1{\left\Vert #1\right\Vert }%

\global\long\def\scr#1{\mathscr{#1}}%

\global\long\def\Im#1{{\rm Im}\, #1}%

\global\long\def\Re#1{{\rm Re}\, #1}%

\global\long\def\one{\overrightarrow{1}}%

\global\long\def\Tr{{\rm Tr}}%

\global\long\def\etc{,\,\dots,\,}%

\global\long\def\a{\alpha}%

\global\long\def\b{\beta}%

\global\long\def\P{\mathbb{P}}%

\global\long\def\E{\mathbb{E}}%

\global\long\def\e{\varepsilon}%

\global\long\def\pr#1#2{\left\langle #1,\,#2\right\rangle }%

\global\long\def\sign{{\rm sign}}%

\global\long\def\t#1{\tilde{#1}}%

\newcommand{\ind}{\mathbbm{1}}

\newcommand{\red}{\color{red}}

\newcommand{\poly}{\mathscr{P}}

\title{Size of Nodal Domains of the eigenvectors \\ of a $G\left(n,\,p\right)$ Graph}

\author{Han Huang \and Mark Rudelson }

\address{Department of Mathematics, University of Michigan, 530 Church St., Ann Arbor, MI 48109, U.S.A.}
\email{\{sthhan, rudelson\}@umich.edu}
\begin{abstract}
Consider an eigenvector of the adjacency matrix of a $G(n,p)$ graph.
A nodal domain is a connected component of the set of vertices where this eigenvector has a constant sign.
It is known that with high probability, there are exactly two nodal domains for each eigenvector corresponding to a non-leading eigenvalue. We prove that with high probability, the sizes of these nodal domains are approximately equal to each other.
\end{abstract}

\maketitle

\section{Introduction}

Nodal domains of the eigenfunctions of the Laplacian on smooth manifolds
have been studied for more than a century. We refer
the readers to the book \cite{Zelditch2017} for the details. If $f:M\to\R$ is such an eigenfunction
on a manifold $M$, then the nodal domain is a connected component
of the set  M where the function $f$ has a constant sign. The number
and the geometry of nodal domains provide an important insight into
the geometric structure of the manifold itself. A classical theorem
of Courant states that the number of nodal domains of the eigenfunction
corresponding to the $k$-th smallest eigenvalue is upper bounded
by $k$, and this number typically grows as $k$ increases \cite{CH89}.
In \cite{DLL11} Dekel, Lee, and Linial pioneered the study of
the nodal domains for graphs. This study was motivated by the usefulness
of the eigenvectors of graphs in a number of partitioning and clustering
algorithms, see \cite{DLL11} and the references therein. In the
last 10 years, these eigenvectors have played a crucial role in many
other computer science problems, including, e.g., community detection
\cite[Section 5.5]{Vershynin2018}. As the Laplacian of a graph is closely
related to the adjacency matrix, Dekel et.al. considered the eigenvectors
of the  latter matrix as an analog of the eigenfunctions of the Laplacian
on a manifold. We will arrange the eigenvectors of the adjacency matrix
in the order corresponding to the decreasing order of the eigenvalues.
An easy variational argument shows that that the first, i.e., the
leading eigenvector has only one domain, so the study of nodal domains
become non-trivial for the non-leading
eigenvectors. In general, one has to
distinguish between the strict and the non-strict domains, where the
former do not include vertices with zero coordinates.

The main result of \cite{DLL11} pertains to the $G(n,p)$ random
graphs in the case when $p\in(0,1)$ is a constant.
 Recall that an Erd\H{o}s-Rényi
Graph $G(n,p)$ is a random graph with $n$ vertices and any two
vertices are connected by an edge with probability $p$ independently.
In this case,
the authors discovered a new phenomenon showing that the behavior
of the number of nodal domains for a $G(n,p)$ graph is essentially
different from that for a manifold. More precisely, they proved that
with probability $1-o(1)$, the two largest non-strict nodal domains
of any non-leading eigenvector contain all but $O_{p}(1)$ vertices,
where the last quantity is uniform over the eigenvectors. Besides
proving this striking result, \cite{DLL11} emphasized that the
main approach to the study of nodal domains is through establishing
delocalization properties of the eigenvectors of random matrices.
At the time \cite{DLL11} was written, the study of delocalization
was in its infancy. Indeed, their theorem relies on a partial case
of \cite[Theorem 3.3]{LPRTV05}, which was the only result available
at that time. As the information on the delocalization of the eigenvectors
grew, so did the knowledge about the finer properties of the nodal
domains. In \cite{Nguyen2017}, Nguyen, Tao, and Vu proved that, with
probability $1-o(1)$, any eigenvector does not have zero coordinates,
which mean that the strong and the weak nodal domains of a $G(n,p)$
graph are the same with high probability. Also, Arora and Bhaskara
\cite{Arora2012EigenvectorsOR} improved the main theorem of \cite{DLL11}
by showing that if $p \ge n^{-1/19+o(1)}$ then with probability $1-o(1)$, any non-leading eigenvector
has exactly two nodal domains.
We refer readers to the 
articles \cite{TVW13, Erdos2013,EKYY12,HLY15,LS18} on other recent developments of local statistics of eigenvalues or 
delocalization of eigenvectors for sparse Erd\H{o}s-Rényi
Graph $G(n,p)$.

After these results became available, Linial put forward a program
of studying the geometry of nodal domains. Considering one of the
domains as earth, and another one as water, one can investigate the
length of the shoreline, which is the boundary of the domains, the
distribution of heights and depths measured as distances to the shoreline,
etc. Unfortunately, this geometry turned out to be trivial in the
case when $p\ge n^{-c}$ for some absolute
constant $c \in (0,1)$.
More precisely, it was proved in \cite{Rudelson2017}
that with probability $1-o(1)$, any vertex in the positive nodal
domain is connected to the negative one, and the same is true for
the vertices in the negative domain. Note that the case of very
sparse graphs $p\le n^{-c}$ is still open and may lead to a non-trivial
geometry. The proof of this result relied on the combination of the
no-gaps delocalization  \cite{Rudelson2016},  and a more classical $\ell_{\infty}$ delocalization
established by Erd\H{o}s, Knowles, Yau, and Yin  \cite{Erdos2013}.
The no-gaps delocalization discussed in
more detail below means that with high probability, any set $S$
of vertices carries a non-negligible proportion of the Euclidean norm
of the eigenvector, and this proportion is bounded below by a function
of $\left|S\right|/n$ only.
 The  $\ell_{\infty}$ delocalization means that the maximal coordinate of any unit eigenvector
does not exceed $n^{-1/2+o(1)}$ with high probability.

In this paper, we establish another natural property of nodal domains.
Namely, we will show that with high probability, the nodal domains
are balanced, i.e. each one of them contains close to $n/2$ vertices
with high probability. Unlike the previous ones, this property does
not follow from the combination of the no-gaps and the $\ell_{\infty}$
delocalization. Indeed, the vector $u\in S^{n-1}$ with $n/3$ coordinates
equal to $\sqrt{2}/\sqrt{n}$ and the rest $n/3$ coordinates equal
to $-1/\sqrt{2n}$ satisfies both properties. Moreover, for such vector,
$\sum_{j=1}^{n}u(j)=0$, so it is orthogonal to the vector $(1/\sqrt{n},\ldots,1/\sqrt{n})$
which is close to the leading eigenvector with high probability.

We prove that the nodal domains are roughly of the same size both
for the bulk and for the edge eigenvectors. However, the methods of proof
in these cases are different. Let us consider the bulk case
first as the proof in this case is shorter. Let $A$ be the adjacency
matrix of $G\left(n,\,p\right)$. We denote eigenvalues of $A$
by $\lambda_{1}\ge\cdots\ge\lambda_{n}$ and the corresponding unit
eigenvectors by $u_{1}\etc u_{n}$.
With a slight abuse of terminology, we will call them the eigenvectors of the graph $G\left(n,\,p\right)$.
\begin{thm}
\label{thm: bulk} (Bulk case) There is $c\in\left(0,\,1\right)$
such that the following holds. 
Let $\eps,\,\kappa\in\left(0,\,1\right)$.
 Let $G\left(n,\,p\right)$ be an Erd\H{o}s-Rényi
Graph with $p\in\left[n^{-c},\,\frac{1}{2}\right]$.
Let $u_{\a}$ be an eigenvector
 of $G\left(n,\,p\right)$ with $\a\in\left[\kappa n,\,n-\kappa n\right]$.
 Denote by $P$ and $N$ the nodal domains of this eigenvector.
Then there exists $\eta=\eta\left(\eps,\,\kappa\right)>0$ such that,
for a sufficiently large $n$,
\[
\P\left(\left|P\right|\vee\left|N\right|\ge\left(\frac{1}{2}+\eps\right)n\right)\le n^{-\eta}.
\]
\end{thm}

The proof relies on \emph{quantum unique ergodicity theorem} for random matrices 
\cite[Theorem 1.1]{bourgade2017} claiming that the distribution of the 
inner product of an eigenvector of $A$ and any vector orthogonal 
to $(1 \etc 1)$ is asymptotically normal.
Readers interested in 
 quantum unique ergodicity are also referred to the articles \cite{BY17, BHY19, BYY19}.
For the edge case, i.e., for the eigenvalues close to the edges of
 the spectrum, the bound similar to \cite[Theorem 1.1]{bourgade2017}
has been established only for the non-sparse regime, i.e.,
 for $p\in (0,1)$ which does not depend of $n$, see \cite{BY17}. 
On the other hand, the gaps between
the eigenvalues near the edges of the spectrum are much larger. The
eigenvalue gap is at least $n^{-2/3-o\left(1\right)}$ for edge eigenvalues
while it is of order $n^{-1-o\left(1\right)}$ for bulk eigenvalues.
Also, the edge eigenvalues enjoy stronger rigidity properties than
the bulk ones.
These facts allow to provide a stronger bound for the size of the nodal domains of an edge eigenvector.
\begin{thm}
\label{thm: edge} (Edge case) Let $G\left(n,\,p\right)$ be an Erd\H{o}s-Rényi
Graph with $p\in\left(0,\,1\right)$.
There exists $\rho=\rho(p)>0$
such that the following holds.
 Let $u_{\a}$ be a n on-leading
eigenvector of $G\left(n,\,p\right)$ with $\min\left\{ \a,\,n-\a\right\} \le\left(\log n\right)^{\rho\log\log n}$.
Denote by $P$ and $N$ the nodal domains of this eigenvector. Then,
for any $\eps>0$, there exists $\delta=\delta(\eps)>0$ independent of $n$ and $p$ such that
\[
\P\left(\left|P\right|\vee\left|N\right|\ge\left(\frac{1}{2}+n^{-\frac{1}{6}+\eps}\right)n\right)\le n^{-\delta}.
\]
for a sufficiently large $n$.
\end{thm}

For a vector $u\in \R^{n}$, let $u( i)$ denote its $i$th component.
Our goal
 in both Theorem \ref{thm: bulk} and \ref{thm: edge}
is to show that with high probability,
\[
\sum_{i=1}^{n}\sign \left(u(i)\right)=o\left(n\right)
\]
for an eigenvector $u$ of $A$. This can be derived by Markov
inequality if
\[
  \EE\left(\sum_{i=1}^{n}\sign\left(u\left(i\right)\right)\right)^{2}=o\left(n^{2}\right).
\]
  The latter equation can be derived if for $i\neq j$,
\begin{equation}
\EE\sign\left(u\left(i\right)u\left(j\right)\right)=o\left(1\right).\label{eq:ExpectationOfEntrySign}
\end{equation}

The proof in both the bulk and the edge case
is aiming to show (\ref{eq:ExpectationOfEntrySign}).
Yet, the approaches are completely different.
The proof in the bulk case relies on
\begin{thm}
  \cite[Theorem 1.5]{Rudelson2016}\label{thm: BHY} Fix arbitrary constants
  $\delta ,\kappa>0$ Let $A$ be an $n\times n$ be the adjacency matrix
  of a $G(n,p)$ graph with $n^{-c}\le p\le 1/2$ for some constant $c>0$.
  For $\epsilon > c_1n^{-1/7}$, every eigenvector $v$ of $A$ satisfies
  $$
    \left(\sum_{i\in I} |v(i)|^2 \right)^{1/2}\ge (c_2\epsilon)^6\norm{v}.
  $$
  for all $I\subset [n]$ with $|I|\ge \epsilon n$.
\end{thm}
and
\begin{thm}
    \cite[Theorem 1.1]{bourgade2017}\label{thm: BHY} Fix arbitrary constants
    $\delta ,\kappa>0$ Let $A$ be an $n\times n$ be the adjacency matrix
    of a $G(n,p)$ graph with $n^{-1+\delta{}}\le p\le1/2$. Let $v_{1}\etc v_{n}$
    be its eigenvectors corresponding to the eigenvalues $\lambda_{1}\ge\ldots\ge\lambda_{n}$.
    For any polynomial $f:\R\to\R$ for any $n\ge n(f)$, $\a\in[\kappa n:n-\kappa n]$
    and any $q\in S^{n-1},q\perp (1 \etc 1)$, there exists an
    $\nu>0$ such that
    \[
    |\E f(n\iprod q{v_{ \a}}^{2})-\E f(g^{2})|\le n^{-\nu}.
    \]
\end{thm}
The last theorem allows to estimate $\EE\sign\left(u\left(i\right)u\left(j\right)\right)$ by replacing $u\left(i\right)$ and $u\left(j\right)$ by independent normal random variables. Yet, this replacement is not straightforward. First, we have to transform the statement of  Theorem \ref{thm: BHY} involving $\iprod q{u}^{2}$ into a one involving $u(i)$ and $u(j)$.
Secondly, and more importantly, we have to approximate the function $\sign(\cdot)$ by a polynomial. Since the polynomial function is unbounded on $\R$, we have to find an approximation which is close to the function $\sign(\cdot)$ point-wise on the set $[-R,R] \setminus (-\delta,\delta)$ with some $0<\delta<1<R$, and at the same time has a controlled growth at infinity.
The latter property is needed to  guarantee that the contribution of the values $u(i) \notin [-R,R]$ does not affect quality of the approximation. The contribution of the values $u(i) \in (-\delta,\delta)$ can be made small by choosing an appropriate $\delta$ due to the no-gaps delocalization.

For the edge case, we represent the adjacency matrix
$A$ in block form:
\[
\left[\begin{array}{cccc}
D &  & W^{\top}\\
\\
W &  & B\\
\\
\end{array}\right]
\]
where $B$ is $n-2$ by $n-2$, $D$ is $2$ by $2$, and $W$ is $n-2$ by $2$.
These matrices are independent. Moreover, using the results of \cite{EJP3054, Erdos2012b, Knowles2013a}, we show that  with high probability, the matrix $B$ has ``typical`` spectral properties.
Relying on the independence of the blocks, it is possible to bound the expectation of $\sign\left(u\left(1\right)u\left(2\right)\right)$ conditioned on the event that $B$ is typical. To use this approach for other pairs of coordinates, we have to show that with high probability, all  $(n-2) \times (n-2)$ principal submatrices of $A$ are typical. This cannot be derived from the union bound since one of the typical properties, namely the level repulsion, holds with probability $1- O(n^{-\delta})$ for some $\delta>0$. To overcome this problem, we condition on the event that the matrix $A$ itself is typical, and show that  on this event, with high probability, all $(n-2) \times (n-2)$ blocks are typical as well.

\subsection{Acknowledgement}

Part of this research was performed while the authors were in residence at the Mathematical Sciences Research Institute  (MSRI) in Berkeley, California, during the Fall semester of 2017,
and at the Institute for Pure and Applied Mathematics (IPAM) in Los Angeles, California, during May and June of 2018.
Both institutions are supported by the National Science Foundation.
Part of this research was performed while the second author visited Weizmann Institute of Science in Rehovot, Israel, where he held Rosy and Max Varon Professorship. We are grateful to all these institutes for their hospitality and for creating an excellent work environment.

The research of the second author was supported in part by the NSF grant DMS 1807316 and by a fellowship from the Simons Foundation.

The  authors are grateful  to Paul Bourgade, Nick Cook, Terence Tao,
and Jun Yin for  fruitful discussions and remarks which led to improved presentation of some of the results.
We also thank the referees for many valuable remarks and suggestions.

\subsection{Notation}
First, $c,c',C,C'$ will denote constants which may change from line to line.
For a positive integer $n$, denote $[n]:=\{1,\,2,\,3,\dots,\,n\}$.
For vectors $u,\, v\in \R^{n}$, let
$\left\| u \right\|_2$ denote the Euclidean norm of $u$,
$\left\|  u\right\|_\infty$ denote the $l_\infty$ norm of $u$,
and $\iprod u{v}$ denote the standard inner product of $u$ and $v$.
The cardinality of a set $S$ will be denoted by $|S|$. For $a,b \in \R$, the notation $a \wedge b$ and $a \vee b$ stands for the minimum and the maximum of $a$ and $b$ respectively.

For a random variable $Z$, we denote its $\psi_2$ norm by $\norm{Z}_{\psi_2}$. The $\psi_2$ norm is defined by the equation
$$
  \EE \exp\left(
    \left(
        \frac{|Z|}{\norm{Z}_{\psi_2}}
    \right)^2
  \right)=2.
$$
We say $Z$ is subgaussian if $\norm{Z}_{\psi_2}$ exists.
By subgaussian vector we mean a random vector with independent
components whose $\psi_2$ norms are uniformly bounded.

Let $\mathbf{Mat}_{sym}(n)$ be the collection of all symmetric
$n\times n$ matrices.
For a symmetric $n\times n$ matrix $H=\{h_{ij}\}_{i,j=1}^n$,
let $\norm{H}$ denote its operater norm,
$\norm{H}_{HS}$ denotes its Hilbert-Schmidt norm.
Precisely,
\[
  \norm{H}_{HS}^2 = \sum_{i,j=1}^n h_{ij}^2 = \sum_{i=1}^n \lambda_i^2,
\]
where $\{\lambda_i\}_{i=1}^n$ are eigenvalues of $H$. Furthermore,
 for $z\in \mathbb{C}$ with $\Im z > 0$,
\[
  G(z) = \frac{1}{H-z}
\]
denote the Green function of $H$, and define the Stieltjes Transform of $H$ by
\[
  m(z) = \frac{1}{n} \Tr(G(z))= \frac{1}{n}\sum_{i=1}^n \frac{1}{\lambda_i-z}
\]
where $\{\lambda_i\}_{i=1}^n$ are eigenvalues of $H$.

Recall the semicircle-law
\[
  \rho_{sc}(x)= \frac{1}{2\pi}\sqrt{\left( 4-x^2 \right)_+},
\]
where $\left( 4-x^2 \right)_+=\max\{4-x^2,\,0\}$. The semicircle law proved in the classical paper of Wigner \cite{Wig1955} is the limit distribution of the empirical distribution
of eigenvalues of Wigner matrices, see e.g. \cite{anderson_guionnet_zeitouni_2009} for the precise formulation and extensions.
The Stieltjes transform of $\rho_{sc}$ is
\[
  m_{sc}(z)= \int_\R \frac{\rho_{sc}(x)}{x-z}\d x.
\]
For a fixed $n$, let $\gamma_i$ be the expected location of $i-$th eigenvalue
(rearranged in a non-increasing order)
according to the semicircle law. That is, $\gamma_i$ satisfies
\[
  \int_{\gamma_i}^2 \rho_{sc}(x) \d x = \frac{i}{n}.
\]
Furthermore, it is easy to check that for $i = o(n)$, we have
\begin{equation}
  \left(\pi\frac{i}{n}\right)^{2/3}
  \le 2-\gamma_i
  \le \left(3\pi\frac{i}{n}\right)^{2/3}.
  \label{eq:gammaalpha}
\end{equation}

\section{Bulk eigenvector}

Consider a graph $G$ with $n$ vertices, and denote by $A$ its adjacency matrix.
  Let
$\lambda_1\ge \lambda_2\ge \cdots \lambda_n$ be the eigenvalues of $A$ and let
$v_\a$ be the unit eigenvector corresponding to $\lambda_\a$.
In order to show that $\sum_{i=1}^n \sign \big( v_\a(i) \big) =o(n)$,
consider a random pair  $(i,j) \subset [n]$ of distinct indices
 which is uniformly chosen among all such pairs. We will check below that if
 $\E \text{sign}(v_{\a}(i)\cdot v_{\a}(j))=o(1)$, then this inequality holds, and
 the nodal domains are of the size close to $n/2$.
We are going to establish this bound for the adjacency matrix of a typical $G(n,p)$ graph.
Since $\sign$ is not a continuous function, it is hard to approach this task directly.
Instead,  we will approximate the function $\sign$ by a suitable
polynomial $f$ and show that
  $\E \big[ f(v_\a(i)\cdot v_\a(j)) \mid A \big]=o(1)$
where the expectation is taken with respect to the random pair $(k,l)$ and  $A$ is the adjacency matrix of a typical $G(n,p)$ graph, i.e., it is chosen from some set of adjacency matrices whose probability is $1-o(1)$.
After that, we will have to estimate the error of this approximation.
 To implement the first step,
we will use Theorem \ref{thm: BHY} to derive a similar bound for
the expectation of an even polynomial of four random coordinates of
the eigenvector.
This will lead to a stronger bound for an even polynomial of two random coordinates.
Finally, applying the latter bound to a one-variable polynomial of the product of two  coordinates, we will get the desired estimate.

 Let us formulate this statement precisely.
Let
  $v_\a \in S^{n-1}$
be a bulk eigenvector of the $G(n,p)$ graph, and let $g_{1}\etc g_{n}\sim N(0,1)$
be independent standard normal random variables.
Denote by
$\E_{(i,j)}$ the expectation with respect to the random
pair of coordinates $(k,l)$, where the matrix $A$ is regarded as  fixed.
\begin{lem}
\label{lem: two coord} Let
 $A,v_\a$
 be as in Theorem \ref{thm: BHY}.
Let $(k,l)$ be a uniformly chosen random pair of elements of $[n]$.
For any \emph{even} polynomial $F:\R^{2}\to\R$, there exists a $\nu>0$
and a set $\mathcal{A}_{F}\in\mathbf{Mat}_{sym}(n)$ such that for
all sufficiently large $n$,
\[
\P(A\in\mathcal{A}_{F})\ge1-n^{-\nu},
\]
and for any $A\in\mathcal{A}_{F}$,
\[
|\E_{(k,l)}F(n^{1/2}
v_\a
(k),n^{1/2}
v_\a(l))-\E F(g_{1},g_{2})|\le n^{-\nu}.
\]
\end{lem}

\begin{proof}
The proof breaks in two parts. First, we will show that the statement
of Theorem \ref{thm: BHY} holds for any $q\in S^{n-1}$ such that
$|\text{supp}(q)|\le4$.  It is enough to prove the statement for $f(x)=x^{d}$.
Without loss of generality, assume that $q=\sum_{j=1}^{4}\a e_{j}$
with $\sum_{j=1}^{4}\a_{j}^{2}=1$. Set $\b:=\iprod{\one}q=n^{-1/2}\sum_{j=1}^{4}\a_{j}$.
Then
\begin{equation}
|\b|\le\frac{4}{\sqrt{n}},\quad q_{0}:=q-\b\overrightarrow{1}\perp\overrightarrow{1}\quad\text{and}\quad\norm{q_{0}}_{2}=1+O(n^{-1/2}).\label{eq: q_0}
\end{equation}
Recall that $w:=\overrightarrow{1}-v_{1}$ satisfies
\begin{equation}
\norm w_{2}\le2\frac{\log n}{\sqrt{n}},\label{eq: first eigenvector}
\end{equation}

see \cite[Theorem 3]{Rudelson2016}.

Let us check that for any $d\in\N$,
\[
\E(n\iprod q{v_{\a}}^{2})^{d}\le C(d)
\]
for some function $C(d)>0$. Indeed, since
$\pr{\overrightarrow{1}}{v_{\a}}=\pr w{v_{\a}}$,
\begin{align*}
\E(n\iprod q{v_{\a}}^{2})^{d}
&=\E(n\iprod{q_{0}+\b\sqrt{n}w}{v_{\a}}^{2})^{d}
\le2^{2d}\left(\E(n\iprod{q_{0}}{v_{\a}}^{2})^{d}+\b^{2d}n^{d}\norm w_{2}^{2d}\right)\\
&\le2^{2d}\left(\E(2g_{1}^{2})^{d}+\left(16\frac{\log^{2}n}{n}\right)^{d}\right)\le C(d).
\end{align*}
where we used (\ref{eq: q_0}), (\ref{eq: first eigenvector}) and
Theorem \ref{thm: BHY} in the second inequality. By Cauchy-Schwarz
inequality, this means that for any $k\in\N$,
\begin{equation}
\E|\sqrt{n}\pr q{v_{\a}}|^{k}\le C'(k).\label{eq: bounded moment}
\end{equation}
Therefore, for any $d\in\N$,
\begin{align*}
\left|\E(n\iprod q{v_{\a}}^{2})^{d}-\E g^{2d}\right|
&\le\left|\E(n\iprod q{v_{ \a}}^{2})^{d}-\E(n\iprod{\frac{q_{0}}{\norm{q_{0}}_{2}}}{v_{ \a}}{}^{2})^{d}\right|+\left|\E(n\iprod{\frac{q_{0}}{\norm{q_{0}}_{2}}}{v_{ \a}}^{2})^{d}-\E g^{2d}\right|\\
&\le\left|\E(n\iprod q{v_{ \a}}^{2})^{d}-\frac{1}{\norm{q_{0}}_{2}^{2d}}\E(n\iprod{q-\b\overrightarrow{1}}{v_{ \a}}^{2})^{d}\right|+n^{-\nu}\\
&\le\left|\E(n\iprod q{v_{ \a}}^{2})^{d}-\E(n\iprod{q-\b w}{v_{ \a}}^{2})^{d}\right|+2n^{-\nu}\\
&\le\sum_{j=1}^{n}\binom{2d}{j}\E|\sqrt{n}\iprod q{v_{ \a}}|^{2d-j}\cdot\left(8\frac{\log n}{\sqrt{n}}\right)^{j}+2n^{-\nu}\le n^{-\nu'}
\end{align*}
for large $n$. Here, the third inequality follows from Theorem \ref{thm: BHY},
the fourth one from (\ref{eq: q_0}) and (\ref{eq: first eigenvector}),
and the last one from (\ref{eq: bounded moment}). This shows that
the conclusion of Theorem \ref{thm: BHY} holds for any $q\in S^{n-1}$
supported on four coordinates. The same argument can be used to prove
this statement for any fixed number of coordinates, but we would not
need it here.

Let us extend the conclusion of Theorem \ref{thm: BHY} to even polynomials
of four variables. Consider an even monomial $G(x_{1}\etc x_{4}):=x_{1}^{d_{1}}\cdot x_{2}^{d_{2}}\cdot x_{3}^{d_{3}}\cdot x_{4}^{d_{4}}$
with $d=d_{1}+d_{2}+d_{3}+d_{4}\in2\N$. Note that for this monomial,
$G(\sqrt{n}v_{\a}(k_{1})\etc\sqrt{n}v_{\a}(k_{4}))$ can be represented
as a finite linear combination of $(\sqrt{n}\pr q{v_{\a}})^{d}$ for
different values of $q\in S^{n-1},\text{supp}(q)\subset\{k_{1}\etc k_{4}\}$.
Hence,

\begin{equation}
\left|\E G(\sqrt{n}v_{\a}(k_{1})\etc\sqrt{n}v_{\a}(k_{4}))-\E G(g_{1}\etc g_{4})\right|\le n^{-\nu}\label{eq: 4 variables}
\end{equation}
and this inequality can be extended to all even polynomials of four
variables.

Now, let $F:\R^{2}\to\R$ be an even polynomial. Let $s\in[\kappa n:n-\kappa n]$.
For a pair $(i,j)\in\binom{[n]}{2}$, define a random variable
\[
Y_{(i,j)}=F(\sqrt{n}v_{ \a}(i),\sqrt{n}v_{\a}(j))-\E F(g_{i},g_{j}),
\]
where $g_{1}\etc g_{n}$ are independent $N(0,1)$ random variables.
Then for any distinct $i,j,k,l,\in[n]$,
\begin{align*}
|\E Y_{(i,j)}Y_{(k,l)}|
&=|\E F(\sqrt{n}v_{\a}(i),\sqrt{n}v_{\a}(j))F(\sqrt{n}v_{ \a}(k),\sqrt{n}v_{\a}(l)) \\
& \quad -\E F(\sqrt{n}v_{\a}(i),\sqrt{n}v_{\a}(j))\E F(g_{k},g_{l})-\E F(g_{i},g_{j})\E F(\sqrt{n}v_{\a}(k),\sqrt{n}v_{\a}(l))\\
& +\E F(g_{i},g_{j})\E F(g_{k},g_{l})|\\
&\le|\E F(g_{i},g_{j})F(g_{k},g_{l})-2\E F(g_{i},g_{j})\cdot\E F(g_{k},g_{l})+\E F(g_{i},g_{j})F(g_{k},g_{l})|
+n^{-\nu} \\
&=n^{-\nu},
\end{align*}
where we used (\ref{eq: 4 variables}) with $G_1(x_{1},x_{2},x_{3},x_{4})=F(x_{1},x_{2})F(x_{3},x_{4}), \ G_2(x_{1},x_{2},x_{3},x_{4})=F(x_{1},x_{2})$, and
$ G_3(x_{1},x_{2},x_{3},x_{4})=F(x_{3},x_{4})$
to derive the inequality.
A similar calculation shows that $|\E Y_{(i,j)}Y_{(k,l)}| =O(1)$ when $i,j,k,l$ are not necessarily distinct.
Hence,
\[
\E\left(\frac{1}{\binom{n}{2}}\sum_{(i,j)\in\binom{[n]}{2}}Y_{(i,j)}\right)^{2}\le\frac{1}{\binom{n}{2}^{2}}\sum_{(i,j,k,l)\in\binom{[n]}{4}}\E Y_{(i,j)}Y_{(k,l)}+O(n^{-1})\le n^{-\nu}.
\]
The Markov inequality implies that there exists a set $\mathcal{A}_{F}'\in\mathbf{Mat}_{sym}(n)$
such that for all sufficiently large $n$,
\[
\P(A\in\mathcal{A}_{F}')\ge1-n^{-\nu/2},
\]
and for any $A\in\mathcal{A}_{F}'$,
\[
\left|\frac{1}{\binom{n}{2}}\sum_{(i,j)\in\binom{[n]}{2}}F(\sqrt{n}v_{ \a}(i),\sqrt{n}v_{\a}(j))-\E F(g_{1},g_{2})\right|
=\left|\frac{1}{\binom{n}{2}}\sum_{(i,j)\in\binom{[n]}{2}}Y_{(i,j)}\right|\le n^{-\nu/4}.
\]
The lemma is proved.
\end{proof}
Applying the previous lemma to a polynomial $F(x,y)=f(x\cdot y)$
for a one-variable polynomial $f$, we derive the following corollary.
\begin{cor}
\label{cor: product of coord} Let $A,v_{\a}$ be as in Theorem \ref{thm: BHY}.
Let $(k,l)$ be a uniformly chosen random pair of elements of $[n]$.
For any polynomial $f:\R\to\R$, there exists a $\nu>0$ and a set
$\mathcal{A}_{f} \subset \mathbf{Mat}_{sym}(n)$ such that for all sufficiently
large $n$,
\[
\P(A\in\mathcal{A}_{f})\ge1-n^{-\nu},
\]
and for any $A\in\mathcal{A}_{f}$,
\[
|\E_{(k,l)}f(nv_{\a}(k)\cdot v_{\a}(l))-\E f(g_{1}g_{2})|\le n^{-\nu}.
\]
\end{cor}

To prove that the nodal domains are balanced, we will use Corollary
\ref{cor: product of coord} with $f$ being an \emph{odd} polynomial approximating
$\text{sign}(x)$ on some interval $[r,R]$. Since $f$ is odd, $\E f(g_{1}g_{2})=0$.
Hence, assuming that the nodal domains are unbalanced, it would be
enough to show that $|\E_{(k,l)}f(nv_{\a}(k)\cdot v_{\a}(l))|$ is non-negligible
to get a contradiction. The values of $r$ and $R$ will be chosen
so that the absolute values of most of the coordinates will fall into
this interval. A simple combinatorial calculation will show that if
the nodal domains are unbalanced, then $\E_{(k,l)}\text{sign}(v_{ \a}(k)\cdot v_{\a}(l))=\Omega(1)$.
Indeed, assume that for a given matrix $A$ and vector $v_j$,
\[
|P| \vee|N| \ge \left( \frac{1}{2} + \eps \right).
\]
Then
\[
 \E_{(k,l)}\text{sign}(v_{\a}(k)\cdot v_{\a}(l))
 = \binom{n}{2}^{-1} \cdot \left[ \binom{|P|}{2} + \binom{|N|}{2}-|P| \cdot |N| \right]
 \ge 4 \eps^2 +O(n^{-1}).
\]
This reduces our task to the comparison between this quantity and
$|\E_{(k,l)}f(nv_{ \a}(k)\cdot v_{\a}(l))|$. To achieve it, we construct
$f$ approximating $\sign(x)$ pointwise on the set $[-R,-r]\cup[r,R]$
and show that the contribution of the coordinates falling outside
of this set is negligible. For the interval $(-r,r)$, this will be
done using the no-gaps delocalization. Handling the set $(-\infty,-R)\cup(R,\infty)$
is more delicate. Since the polynomial is unbounded on this set, we
will control the $L_{2}$ norm of $f$ and use the Markov inequality.
This argument requires constructing the polynomial $f$ which approximates
$\text{sign}(x)$ in two metrics simultaneously: uniformly on the
set $[-R,-r]\cup[r,R]$ and in $L_{2}(\mu)$ norm on $\R$. The measure
$\mu$ here will be the probability measure on $\R$ defined by
\[
\mu(B)=\P(g_{1}g_{2}\in B).
\]

Instead of controlling two metrics at the same time, we will introduce one Sobolev norm which will be stronger than both metrics.
Such norm can be chosen in many different ways.
We will chose a particular way which makes the argument shorter.

Let $\eta:\R \setminus \{0\}\to(0,\infty)$ and $\psi:\R\to(0,\infty)$ be even
functions such that
\begin{itemize}
\item $\eta\in C^{1}((0,\infty))$, $\psi\in C^{1}(\R)$;

\item $\eta(x),\psi(x)=\frac{1}{\pi}\exp(-x/2)$ for all $x\ge2$;

\item $\eta(x)\ge\phi(x)$ for all $x>0$, and $\eta \in L_1(\R)$.
\end{itemize}
Consider a weighted Sobolev space $H$ defined as the completion of
the space of $C^{1}(\R)$ functions for which the norm
\[
\norm f_{H}^{2}:=\int_{\R}f^{2}(x)\eta(x)\,dx+\int_{\R}(f'(x))^{2}\psi(x)\,dx
\]
is finite. Note that $H\subset C\left(\R\right)$. Indeed, for any
$M>0$, $a<b,\,a,b\in[-M,M]$ and any $f\in C^{1}\left(\R\right)$,
\begin{align} \label{eq: (1/2)-Lipschtz}
|f(b)-f(a)|
&=\left|\int_{a}^{b}f'(a)\,dx\right|\le\left(\min_{x\in[-M,M]}\psi(x)\right)^{-1}\cdot\int_{a}^{b}|f'(x)|\psi(x)\,dx \notag\\
&\le\left(\min_{x\in[-M,M]}\psi(x)\right)^{-1}\cdot\left(\int_{a}^{b}(f'(x))^{2}\psi(x)\,dx\right)^{1/2}\left(\int_{a}^{b}\psi(x)\,dx\right)^{1/2}\\
&\le\left(\min_{x\in[-M,M]}\psi(x)\right)^{-1}\cdot\norm f_{H}\cdot\left(\max_{x\in[-M,M]}\psi(x)\right)^{1/2}\cdot(b-a)^{1/2},\label{eq: cont} \notag
\end{align}
and the same inequality holds for the completion.

We will need the following lemma.
\begin{lem} \label{lem: density}
Let $h\in C^{1}(\R)$ be an odd function such that $\norm h_{\infty}+\norm{h'}_{\infty}<\infty$.
Then  for any $\delta>0$, there exists an odd polynomial $Q$ satisfying
$\norm{Q-h}_H<\delta$.
\end{lem}
\begin{proof}
Denote by $\poly$ the set of all polynomials.
Let $E_{odd}$ be the
set of all odd functions $h\in C^{1}(\R)$ such that $\norm h_{\infty}+\norm{h'}_{\infty}<\infty$.
It is enough to prove that $E_{odd}\subset\text{Cl}_{H}( \poly)$.
Indeed, if this is proved,
then for any $\delta >0$ there exists $q\in \poly$ such that
$
\norm{h-q}_{H}<\delta .
$
 Setting $Q(x)=\frac{1}{2}(q(x)-q(-x))$
to make the polynomial odd would finish the proof.

Assume to the contrary that $E_{odd}\not\subset\text{Cl}_{H}(\poly)$.
Then there exists $h\in\text{Cl}_{H}(E_{odd})\setminus\{0\}$ such
that $\pr h{x^{n}}_{H}=0$ for any $n\in\{0\}\cup\N$. We will prove
that this assumption leads to a contradiction. To this end, set
\[
F(z)=\int_{\R}h(x)e^{zx}\eta(x)\,dx+\int_{\R}h'(x)ze^{zx}\psi(x)\,dx.
\]
Using the Cauchy-Schwarz inequality, one can check that the function
$F$ is analytic in $\{z:\,|\text{Re}(z)|<1/2\}$ and
\[
F^{(n)}(0)=\int_{\R}h(x)x^{n}\eta(x)\,dx+\int_{\R}h'(x)nx^{n-1}\psi(x)\,dx=\pr h{x^{n}}_{H}=0.
\]
Hence, $F(z)=0$, and applying this conclusion to $z=it,\,t\in\R$,
we see that $h$ satisfies the equality
\begin{align*}
\big(h\eta-(h'\psi)'\big)^{\wedge}=0\ \text{ and thus } h\eta-(h'\psi)'=0
\end{align*}
in the sense of distributions where $(\cdot)^{\wedge}$ denotes the
Fourier Transform. Since the function $h\eta$ is continuous
on $(0,\infty)$, $h$ satisfies the differential equation
\begin{equation}
h(x)\eta(x)-(h'(x)\psi(x))'=0\label{eq: ODE}
\end{equation}
pointwise for all $x\in(0,\infty)$. This in turn means that $h''$ is well-defined
on $(0,\infty)$. Actually, with a little effort, one can prove that
this differential equation is satisfied for all $x\in\R$, but we would not need it
for our proof.

Since $h\in\text{Cl}_{H}(E_{odd})$, $h$ is an odd continuous
function. For $x\ge2$, (\ref{eq: ODE}) reads
\[
h(x)+\frac{1}{2}h'(x)-h''(x)=0,
\]
and so $h(x)=C_1\exp(\lambda_1 x) + C_2 \exp(\lambda_2 x)$ with
\[
\lambda_1=\frac{1-\sqrt{17}}{4}, \quad \lambda_2=\frac{1+\sqrt{17}}{4}
\]
 for
all $x\ge 2$.
Since $\lambda_2>1/2$ and $h \in H$, $C_2=0$.
Without loss of generality, assume that $h(2)>0$, i.e.,
$C_1>0$. Then $h'(2)<0$ and since $h(0)=0,\,h(2)>0$, there exists
$x\in(0,2)$ such that $h'(x)>0$. Denote
\[
x_{0}=\sup\{x\in(0,2):\,h'(x)>0\}.
\]
Then $h'(x_{0})=0$ and since $h'(x)\le0$ for $x>x_{0}$, we have
$h(x_{0})>0$. Hence, (\ref{eq: ODE}) implies that $h''(x_{0})>0$.
Therefore $h'(x)>0$ for some $x>x_{0}$, which contradicts the definition
of $x_{0}$. This contradiction finishes the
proof of the lemma.
\end{proof}

We are now ready to prove the main result of this section.

\begin{proof}[Proof of Theorem \ref{thm: bulk}]
Fix an $\e>0$, and let $\Omega$ be the event that $|P| \vee |N| \ge (1/2+\e)n$.
Let $(k,l)$ be a uniformly chosen random pair of distinct elements of $[n]$.
Assume that $\Omega$ occurs.
Then
\begin{equation} \label{eq: bal 11}
 \P (v(k)v(l) >0 \mid A)
 \ge \frac{ \binom{(1/2+\e)n}{2}+\binom{(1/2-\e)n}{2}}{\binom{n}{2}}
 =\frac{1}{2}+2\e^2 +O(n^{-1})
\end{equation}
and
\begin{equation} \label{eq: bal 12}
 \P (v(k)v(l) <0 \mid A)
 \le \frac{ \left( \frac{1}{4}-\e^2 \right) n^2}{\binom{n}{2}}
 =\frac{1}{2}-2\e^2 +O(n^{-1}).
\end{equation}

By the no-gap delocalization theorem \cite[Theorem 1.5]{Rudelson2016}, for $r=c \e^{22}$,
\[
 \P \left( |\{j \in [n]: \ |v(j)| \le r^{1/2} n^{-1/2} \}| \ge (\e^2 /8) n \right) \le \exp(-c \e n).
\]
Let $\Omega_{large}$ be the event that  $|\{j \in [n]: \ |v(j)| \le r^{1/2} n^{-1/2} \}| \le (\e^2 /8) n$,
and assume that $\Omega \cap \Omega_{large}$ occurs.
Then
\begin{equation} \label{eq: bal 2}
 \P( n |v(k)| \cdot |v(l)| \le r \mid A) \\
  \le \P (|v(k)| \wedge |v(l)| <r^{1/2} n^{-1/2} \mid A)
 \le 1- \frac{\binom{(1-(\e^2 /8) )n}{2}}{\binom{n}{2}} \le \frac{\e^2}{4}.
\end{equation}
Let $R \ge (c_0 \e)^{-4}$, where the constant $c_0>0$ will be chosen later.
Since $\norm{v}_2=1$,
\[
 |\{j \in [n]: \ |v(j)| \ge R^{1/2}n^{-1/2} \}
 \le \frac{n}{R} \le (c_0 \e)^{4} n,
\]
so
\begin{equation} \label{eq: bal 3}
 \P ( n |v(k)| \cdot  |v(l)| \ge R \mid A ) \\
 \le \P (|v(k)| \ge  R^{1/2}n^{-1/2} \ \text{or } \ |v(l)| \ge  R^{1/2}n^{-1/2} \mid A )
 \le 2 (c_0 \e)^{4}.
\end{equation}
Summarizing \eqref{eq: bal 11}, \eqref{eq: bal 12},  \eqref{eq: bal 2}, and \eqref{eq: bal 3}, and choosing $c_0$ small enough, we conclude that on the event $\Omega \cap \Omega_{large}$,
\begin{align*}
  \P (n v(k)v(l)  \in [r, R] \mid A)
  & \ge \frac{1}{2} + \frac{3}{2} \e^2 +O(n^{-1})
  \intertext{and}
  \P (n v(k)v(l) \in [-r, -R] \mid A) & \le \frac{1}{2} - \frac{3}{2} \e^2 +O(n^{-1}).
\end{align*}

Let $h \in C^{\infty}(\R)$ be an odd function such that $h(x)=\sign(x)$ for any $x \notin (-r,r)$.
 Lemma \ref{lem: density} and inequality \eqref{eq: (1/2)-Lipschtz} imply that there exists an odd polynomial $Q$ such that $\norm{h-Q}_{L_2( \phi \, dx)}<\e$ and
\[
  \max_{x \in [-R,R]} |h(x)-Q(x)| \le \frac{\e^2}{2}.
\]
By Corollary \ref{cor: product of coord}, there exists $\mathcal{A}_{Q}$ with $\P(A \in \mathcal{A}_{Q}) \ge 1- n^{-\nu}$ such that for any $A \in \mathcal{A}_{Q}$,
\[
 \E_{(k,l)} Q(n v(k)v(l))
  \le \E Q (g_1 g_2) + n^{-\nu}
 = n^{-\nu},
\]
for sufficiently large $n$, since the polynomial $Q$ is odd.
We will provide a lower estimate of this expectation in terms of $\P(\Omega)$.
We have
\[
 \E_{(k,l)} Q(v(k)v(l))
 =\E_{(k,l)} Q(nv(k)v(l)) \cdot \mathbf{1}_{n|v(k)v(l)| \le R} +
 \E_{(k,l)} Q(nv(k)v(l)) \cdot \mathbf{1}_{n|v(k)v(l)| > R}.
\]
Let us estimate these terms  separately.
On the event $\Omega \cap \Omega_{large}$,
\begin{align*}
 \E [ Q(nv(k)v(l)) \cdot \mathbf{1}_{n|v(k)v(l)| \le R} \mid A ]
& \ge \left( 1-\frac{\e^2}{2} \right) \P (n v(k)v(l) \in [r,R] \mid A ) \\
 &- \left( 1+\frac{\e^2}{2} \right) \P (n v(k)v(l) \in [-R,-r] \mid A ) \\
 &- \left( 1+\frac{\e^2}{2} \right) \P (n v(k)v(l) \in [-r,r] \mid A ) \\
& \ge 2 \e^2+O(n^{-1}).
\end{align*}

If $A \in  \mathcal{A}_{Q^2}$, then
\[
 \E [Q^2(n v(k)v(l)) \mid A]
 \le \E Q^2 (g_1 g_2) +n^{-\nu} \\
 \le \left( \norm{f}_{L_2(\phi \, dx)} + \e \right)^2 +n^{-\nu}
 \le C.
\]
Hence, by \eqref{eq: bal 3} and Cauchy-Schwarz inequality, for any $A \in  \mathcal{A}_{Q^2}$,
\begin{align*}
 \E \left[ Q(n v(k) v(l)) \cdot \mathbf{1}_{n|v(k)v(l)| > R} \mid A \right]
 &\le \left( \P [n|v(k)v(l)| \ge R \mid A] \right)^{1/2} \cdot \left( E [Q^2(n v(k)v(l)) \mid A]  \right)^{1/2} \\
 &\le C (c_0 \e)^2
 \le \frac{\e^2}{2}
\end{align*}
if $c_0$ is chosen sufficiently small.
Thus, if $A \in  \mathcal{A}_{Q^2}$ and the event  $\Omega \cap \Omega_{large}$ occurs and $n$ is sufficiently large to absorb the $O(n^{-1})$  term, then
\[
 \E \left[ Q(n v(k) v(l)) \mid A \right] \ge \frac{\e^2}{4},
\]
and so, $A \notin  \mathcal{A}_{Q}$.
This means that $\Omega \cap \Omega_{large} \cap \{A \in \mathcal{A}_{Q^2} \cap \mathcal{A}_{Q} \} = \varnothing$, and so
\[
 \P(\Omega) \le \P(\Omega_{large}^c)+ \P(A \in \mathcal{A}_{Q^2}^c)+ \P(A \in \mathcal{A}_{Q}^c)
 \le n^{-\nu}.
\]
The theorem is proved.
\end{proof}

\section{Edge Eigenvector}

Let $A$ be the adjacency matrix of a $G\left(n,\,p\right)$
graph with a fixed $p\in\text{\ensuremath{\left(0,\,1\right)}}$.
Denote by $u$  a non-leading edge eigenvector.
We are aiming to show that
\begin{equation}
\EE\left(\sign\left(u\left(1\right)u\left(2\right)\right)\right)\le n^{-1/3+\eps}\label{eq:EDGE_Expectation_Sign}
\end{equation}
for a sufficiently small $\eps>0$. If proved, it leads to
\[
\EE\left(\sum_{i}\sign u\left(i\right)\right)^{2}=n+\sum_{i\neq j}\EE\sign\left(u\left(i\right)u\left(j\right)\right)\le n+{n \choose 2}n^{-1/3+\eps}\le n^{5/3+\eps},
\]
 because $u(i)u(j)$  has the same distribution as $u(1)u(2)$
for all $i\neq j$ due to the i.i.d. property of the entries the matrix.

Then, by Markov's inequality, we can derive a bound for $\PP\left(\left|\sum_{i}\sign u\left(i\right)\right|\ge n^{5/6+\eps}\right)$
and thus prove Theorem \ref{thm: edge}. Due to technical difficulties,
we would not derive (\ref{eq:EDGE_Expectation_Sign}) directly. Instead,
we find an event $\scr A$ so that
\begin{equation}
\EE\left(
  \sign\left(u\left(1\right)u\left(2\right)\right)
  \,|\,\scr A\right)\le n^{-1/3+\eps}.
\label{eq:signu1u2}
\end{equation}

The event $\scr A$ will be constructed so that $\PP\left(\scr A^{c}\right)\le n^{-\delta}$
where $\delta>0$ may depend on $\eps$. In view of the estimate above,  we have
\begin{align*}
\PP\left(\left|\sum_{i}\sign u\left(i\right)\right|\ge n^{5/6+\eps/2}\right) & \le\PP\left(\scr A^{c}\right)+\PP\left(\left|\sum_{i}\sign u\left(i\right)\right|\ge n^{5/6+\eps}\,|\,\scr A\right)\\
 & \le n^{-\delta}+n^{-\eps}\le n^{-\delta'},
\end{align*}
which finishes the proof of Theorem (\ref{thm: edge}).

Up to a scaling, $A$ is a Wigner matrix with two deterministic shifts:
\begin{equation}
\sqrt{\frac{1}{p(1-p)n}}A =H+\sqrt{\frac{pn}{1-p}}\one\one^{\top}
-\sqrt{\frac{p}{(1-p)n}}I_n \label{eq:ANormalization}
\end{equation}
where $H_{ij}=\left(h_{ij}\right)$ is a symmetric matrix with $0$ diagonal, i.i.d entries $h_{ij}$ with mean $0$ and variance $1/n$ above the diagonal:
\begin{equation}
h_{ij}=\begin{cases}
\sqrt{\frac{1-p}{p}}\frac{1}{\sqrt{n}} & \text{with probability }p,\\
-\sqrt{\frac{p}{1-p}}\frac{1}{\sqrt{n}} & \text{with probability }1-p,
\end{cases}\label{eq:H_EntryDistribution-1}
\end{equation}
and $\one\in S^{n-1}$ is the vector such that every component equals $\frac{1}{\sqrt{n}}$. Notice that the last term in \eqref{eq:ANormalization}
does not affect the eigenvectors and the order of eigenvalues of $\sqrt{\frac{1}{p(1-p)n}}{A}$.
Therefore, it is sufficient to prove \eqref{eq:signu1u2}
for the non-leading edge eigenvectors of
\begin{equation}
  \t {A} :=H+\sqrt{\frac{pn}{1-p}}\one\one^{\top}.
  \label{eq:deftAp}
\end{equation}
Furthermore, we will only prove the theorem for the
eigenvectors belonging to the positive edge $\{u_\a:\,\a\le \varphi_n^\rho\}$. The proof
for eigenvectors $\{u_\a:\,n-\a\le \varphi_n^\rho\}$
is essentially the same.

\subsection{Outline of the proof}

To lighten the notation, assume that $A$ is an $(n+2) \times (n+2)$ matrix.

 It is convenient to break the matrix $\t {
  A}$ into the blocks:
\begin{equation}
\t { A}=\left[\begin{array}{cccc}
D &  & W^{\top}\\
\\
W &  & B\\
\\
\end{array}\right], \label{eq:ABlockDecomposition}
\end{equation}
where $B$ is of size $n\times n$ and $D$ is of size $2\times2$.
Let $G\left(z\right):=\frac{1}{B-z}$ be the Green function of $B$.
We will write the eigenvalues of $\t {A}$ in terms of $B,\,W$ and $D$:
\begin{prop}
Any $\lambda \in \R$ satisfying
\begin{equation}
\det\left(W^{\top}G\left(\lambda\right)W-D+\lambda I_{2}\right)=0\label{eq:EdgeEvalueDetection}
\end{equation}
is an eigenvalue  of $\t {A}$.
Furthermore, let $q\in\R^{2}$
be a non-trivial null vector of $W^{\top}G\left(\lambda\right)W-D+\lambda I_{2}$.
Then, $\left[\begin{array}{c}
q\\
-G\left(\lambda\right)Wq
\end{array}\right]$ is an eigenvector corresponding to $\lambda$.
\end{prop}

\begin{proof}
Assume that
\[
\det\left(W^{\top}G\left(\lambda\right)W-D+\lambda I_{2}\right)=0.
\]

Let $q\in\R^{2}$ be a non-trivial null vector of $W^{\top}G\left(\lambda\right)W-D+\lambda I_{2}$.
Then, we have
\[
\left[\begin{array}{cccc}
D-\lambda &  & W^{\top}\\
\\
W &  & B-\lambda\\
\\
\end{array}\right]\left[\begin{array}{c}
q\\
-G\left(\lambda\right)Wq
\end{array}\right]=\vec{0}.
\]

Therefore,  $\lambda$ is an eigenvalue of $\t {A}$ and $u=\left(\begin{array}{c}
q\\
-G\left(\lambda\right)Wq
\end{array}\right)$ is the corresponding eigenvector.
\end{proof}
Up to a scaling, we have $q=\left[\begin{array}{c}
1\\
-\frac{w_{1}^{\top}G\left(\lambda\right)w_{1}-d_{11}+\lambda}{w_{1}^{\top}G\left(\lambda\right)w_{2}-d_{12}}
\end{array}\right]$ where $w_{1},\,w_{2}$ are the column vectors of $W$ and $D=\left[\begin{array}{cc}
d_{11} & d_{12}\\
d_{12} & d_{22}
\end{array}\right]$. Therefore,
\begin{equation}
\sign\left(u\left(1\right)u\left(2\right)\right)=\sign\left(-\frac{w_{1}^{\top}G\left(\lambda\right)w_{1}-d_{11}+\lambda}{w_{1}^{\top}G\left(\lambda\right)w_{2}-d_{12}}\right).\label{eq:EdgeSign}
\end{equation}

Our goal is to estimate $\EE\sign\left(-\frac{w_{1}^{\top}G\left(\lambda\right)w_{1}-d_{11}+\lambda}{w_{1}^{\top}G\left(\lambda\right)w_{2}-d_{12}}\right)$.
To this end, we would like to take advantage of independence of $B,W$, and $D$.
However, the fact that $\lambda$ depends on all these random quantities precludes us from using this independence straightforwardly.
 This forces us to consider
\begin{align*}
s\left(E\right) & :=\sign\left(-\frac{w_{1}^{\top}G\left(E\right)w_{1}-d_{11}+E}{w_{1}^{\top}G\left(E\right)w_{2}-d_{12}}\right)
\end{align*}
for a constant $E$ instead on dealing with $\lambda$ directly.
To analyze the behavior of the function $s$, it is necessary
to know what the matrix $B$ looks like.

Let $\left\{ \mu_{\a}\right\} _{\a=1}^{n}$ be the eigenvalues of
$B$ arranged in a non-increasing order and let $\left\{ u_{\a}\right\} _{\a=1}^{n}$
be the corresponding unit eigenvectors. Observe that, up to a scaling
factor $\sqrt{\frac{n+2}{n}}$, $B$ is a Wigner matrix with a rank
1 shift:
\[
B=M+\sqrt{\frac{p\left(n+2\right)}{\left(1-p\right)}}ll^{\top},
\]
where $M$ is the lower right $n$ by $n$ minor of $H$
 (from \eqref{eq:ANormalization} and \eqref{eq:H_EntryDistribution-1}), and $l\in\R^{n}$
is the vector with all its components equal to $\frac{1}{\sqrt{n+2}}$.
Here, $\sqrt{\frac{n+2}{n}}M$ is a generalized Wigner matrix having nice spectral properties with high probability.

The proof of Theorem \ref{thm: edge} breaks into 4 steps:

\vskip 0.1in
\emph{1. Typical spectral properties of $M$.}

Here we are encountering the first obstacle. We want to fix a typical sample $M$ to compute $s(E)$. In particular, we want
this sample to have gaps between the eigenvalues close to the edge of order at least $n^{-2/3-\eps}$. Such property is called level repulsion in the edge:
\begin{condition}
\label{cond:LevelRepulsionEdge}(Level Repulsion on Edge) A random
Hermitian matrix $H$ is said to satisfy level repulsion at the edge,
if for any $C_{LR}>0$, and $\eps_{LR}>0$, there exists
$\delta_{LR}>0$, with probability at least $1-n^{-\delta_{LR}}$
\begin{equation}
\max_{E\sub\left[2-n^{-2/3}\varphi_n^{C_{LR}},\,2+n^{-2/3}\varphi_n^{C_{LR}}\right]}\scr N\left(E-n^{-2/3-\eps_{LR}},\,E+n^{-2/3-\eps_{LR}}\right)<2.\label{eq:levelRepulsion}
\end{equation}
\end{condition}
We remark that it is known that a GOE (Gaussian Orthogonal Ensemble) matrix model satisfy this condition, and we will show in appendix that our matrices $H$ and $M$  satisfy this condition as well.

Notice that such level repulsion is achievable with high probability for a single $n \times n$ principal minor $M$, but we need it for all minors simultaneously, and  the probability estimate too weak to be combined with the union bound.
Instead, we define $\scr{A}$ as the event that the
 $(n+2) \times(n+2)$ matrix $H$ has the desired spectral properties.
 In this case, $\scr{A}$ is likely in a sense that
 $\PP({\scr A}^c)<n^{-\delta}$ for some $\delta>0$.
However, we cannot condition on $\scr{A}$ directly
as in this way we will lose the independence of $B$,
$W$, and $D$ while estimating $s(E)$.
 Therefore, in the first step we will define the event $\scr{A}$ and show that
\[
\EE \left( |\ind_{H \text{ is typical}}- \ind_{M \text{ is typical}}|\right) \text{ is small enough.}
\]
This would allow us to use independence while conditioning on the event that $M$ is typical and avoid invoking the union bound while applying this argument to all $n \times n$ principal minors.

\vskip 0.1in
\emph{2. From spectral properties of $M$ to spectral properties of $B$.}

In the second step, we fix a typical $M$, and consider the spectral properties
of its rank one perturbation $B$.  We expect $B$ to behave like $M$ with an exceptional
eigenvector almost parallel to $l$ and the corresponding eigenvalue close to $\sqrt{\frac{p\left(n+2\right)}{1-p}}$.
We will quantify these properties in Definition  \ref{cond:Btypical}
in Section \ref{subsec:Typical B}.

\vskip 0.1in
\emph{3. Concentration of $w_{i}^{\top}G\left(E\right)w_{j}-d_{ij}+E$.}

The expression above is a key quantity in analyzing $s(E)$. To bound $s(\lambda)$ for $\lambda$ being an edge eigenvalue of $\tilde{A}_p$, we have to
understand the behavior of $s(E)$ for different $E$. To this end, we derive the concentration of $w_{i}^{\top}G\left(E\right)w_{j}$
for $i,\,j\in\left\{ 1,\,2\right\} $. By definition,
\[
w_{i}^{\top}G\left(E\right)w_{j}=\sum_{\a\in\left[n\right]}\frac{1}{\mu_{\a}-E}\iprod{w_{i}}{u_{\a}}\iprod{w_{j}}{u_{\a}}.
\]
If $E$ is much closer to an eigenvalue $\mu_{\a_{E}}$ than any other
eigenvalues, then, we expect  $w_{i}^{\top}G\left(E\right)w_{j}$
to be dominated by the term $\frac{1}{\mu_{\a_{E}}-E}\iprod{w_{i}}{u_{\a_{E}}}\iprod{w_{j}}{u_{\a_{E}}}$.
We will show that after conditioning on a typical $B$, with high probability in $W$ and $D$ we have
\begin{equation}
\forall i,\,j\in\left\{ 1,\,2\right\}
\;w_{i}^{\top}G\left(E\right)w_{j}\simeq-\delta_{ij}+\frac{\iprod{w_{i}}{u_{\a_{E}}}\iprod{w_{j}}{u_{\a_{E}}}}{\mu_{\a_{E}}-E}
\end{equation}

\vskip 0.1in
\emph{4. Completion of the proof.}

We combine the results obtained at previous steps to show that
\[
\EE\left(s\left(\lambda_\alpha\right) \ind_{\text{H is typical}}\right)=n^{-1/3+C\eps_{LR}}.
\]
Once this estimate is proved, the main theorem follows immediately.

\subsection{\label{subsec:Typical M} A typical sample of $M$ }

Let $M$ be an $n \times n$ principal submatrix of $H$.
Let $\left\{ \nu_{\a}\right\} _{\a=1}^{n}$ be the eigenvalues of $M$ arranged
in a non-increasing order and let $\left\{ v_{\a}\right\} _{\a=1}^{n}$
be the corresponding unit eigenvectors. Let $G_{M}\left(z\right):=\left(M-z\right)^{-1}$
be the Green function of $M$ and
\[
m_{M}\left(z\right):=\frac{1}{n} \Tr\left(G(z)\right)= \frac{1}{n}\sum_{\a=1}^{n}\frac{1}{\nu_{\a}-z}
\]
be the Stieltjes transform of $M$.

A special role in the proof will be played by the level repulsion property, and the strength of the level repulsion has to be carefully chosen for matrices of different sizes.
Let $t>0$. We will say that an $m \times m$ symmetric matrix $B$ satisfies the level repulsion property with parameter $t$ if
for any two distinct eigenvalues $\nu,\,\nu'$ of $A$ in
$\left[2-n^{-2/3}\varphi_n^{3\rho},\,2+n^{-2/3}\varphi_n^{3\rho}\right]$,
we have
\[
\left|\nu-\nu'\right|>t.
\]
  In the argument below, $m$ takes values from $n-4$ to $n$.
Denote the set of such matrices by $\scr{LR}(n,t)$. Lemma \ref{lem:change t} asserts that
\[
\PP(M \in \scr{LR}(n, n^{-2/3-\eps_{LR}}) ) \ge 1- n^{-\delta_{LR}}
\]
for some $\delta_{LR}>0$.
We start with a lemma showing that the parameter $t$ in the definition of level repulsion can be adjusted without significantly changing this probability.
\begin{lem}  \label{lem:change t}
Let $C>0$.
 Let $M$ be an $n \times n$ symmetric random matrix.
 There exists $\theta \in (1/2,1)$
 which depends on the distribution of $M$ 
  such that
 \[
\PP \left(M \in \scr{LR} \left(n, \theta n^{-2/3-\eps_{LR}} -4 \frac{\varphi_n^C}{n} \right) \right)
- \PP \left(M \in \scr{LR}(n,  \theta n^{-2/3-\eps_{LR}} ) \right)
\le n^{-1/3+ 2 \e_{LR}}.
\] 
\end{lem}
\begin{proof}
For $k \ge 0$, denote
\[
P_k:= \PP \left(M \in \scr{LR} \left(n, n^{-2/3-\eps_{LR}} -k \frac{\varphi_n^C }{n} \right) \right).
\] 
Then $P_k \in (0,1)$ form an increasing sequence. Hence, there exists $k \le  4 n^{1/3 - 2\eps_{LR}}$ such that
\[
 P_{k+4}-P_k \le n^{-1/3+ 2 \e_{LR}}.
\]
This implies the lemma if we choose $\theta$ so that
$ \theta n^{-2/3-\eps_{LR}}=   n^{-2/3-\eps_{LR}} -k \frac{\varphi^C_n}{n}$ and note that $\theta>1/2$.
\end{proof}
We will fix this value of $\theta$ for matrices $H$ whose entries are distributed as in \eqref{eq:H_EntryDistribution-1}  for the rest of the proof.

Let us collect the properties of the $n \times n$ submatrices of $H$ which we will use throughout the proof.
\begin{definition} \label{def: Ank}
Fix  $\eps_{LR}>0$ and $\rho>1$, set
\begin{equation}
    \eta = n^{-2/3-2\eps_{LR}}. \label{eta} 
\end{equation}

  Denote by $\scr A_{(n,k)}$ the set of symmetric $n \times n$ matrices  $M$ having the following properties:
\begin{itemize}
\item Isotropic local semicircular law:
\begin{equation}
\sup_{\left|E-2\right|\le n^{-2/3+3\eps_{LR}}}\sup_{x,y\in\left\{ e_{i}\right\} _{i=1}^{n}\cup\left\{ l\right\} }\left|\iprod x{G_{M}\left(E+\i\eta\right)y}-\iprod xym_{sc}\left(E+\i\eta\right)\right|<3n^{-\frac{1}{3}+3\eps_{LR}},
\label{eq:IsotropicGreenFunction}
\end{equation}
\item  Rigidity of eigenvalues:
\begin{equation}
\left|\nu_{\a}-\gamma_\a\right|\le\varphi_n^{C_{re}}\left[\min\left(\a,\,n-\a+1\right)\right]^{-1/3}n^{-2/3},\label{eq:M_rigidity}
\end{equation}
where $C_{re}>1$ is a universal constant, and
$\gamma_{\a}$ satisfies
$\int_{\gamma_{\a}}^{2}\frac{2}{\pi}\sqrt{4-x^{2}}dx=\frac{\a}{n}$.
\item $l_{\infty}$-delocalization of eigenvectors:
\begin{equation}
\forall\a,\,\norm{v_{\a}}_{\infty}\le\frac{\varphi_n^{C}}{\sqrt{n}},\label{eq:M_evectorl_infinity}
\end{equation}
\item Isotropic delocalization of eigenvectors:
\begin{align}
\max_{\a\in\left[n\right]}\left|\iprod{v_{\a}}l\right|^{2} & <n^{\eps_{LR}-1},
\label{eq:IsotropicDelocalization}
\end{align}

\item Level repulsion at the edge:
$M \in \scr{LR} \left(n, \theta n^{-2/3-\eps_{LR}} -k \frac{\varphi_n^C}{n} \right)$, i.e., \\
for any two distinct eigenvalues $\nu,\,\nu'$ of $M$ in
$\left[2-n^{-2/3}\varphi_n^{3\rho},\,2+n^{-2/3}\varphi_n^{3\rho}\right]$,
we have
\begin{equation} \label{eq:LevelRepulsion}
\left|\nu-\nu'\right|>\theta n^{-2/3-\eps_{LR}} -k \frac{\varphi_n^C}{n}. 
\end{equation}
The value of $\theta$ is chosen  to satisfy the condition of Lemma \ref{lem:change t}.
\end{itemize}
\end{definition}
A typical Wigner matrix belongs to the set $\scr A_{(n,0)}$,
see \cite{Erdos2012b}, \cite{EJP3054}.
However, we need this fact not for a single matrix $M$, but for all $n \times n$ principal submatrices of the $(n+2) \times (n+2)$ matrix $H$.
Denote by $H^{(k)}$ the $(n+1) \times (n+1)$ principal submatrix of $H$ with row and column $k$ removed.
Similarly, denote by $H^{(i, j)}$ the $n \times n$ principal submatrix of $H$ with rows and columns $i,j$ removed.
 The properties  \eqref{eq:IsotropicGreenFunction} -- \eqref{eq:M_evectorl_infinity} hold with an overwhelming probability, which allows to use a union bound while establishing them.
In contrast to it, property \eqref{eq:LevelRepulsion} holds only with probability $1-n^{-\delta_{LR}}$ for some $\delta_{LR}>0$, which is too weak to be combined with the union bound.
To guarantee that the level repulsion holds with high probability for all principal submatrices, we show that the eigenvalues of these submatrices are located closely to the eigenvalues of the original matrix.
To this end, we need the following lemma.
\begin{lem}
\label{lem:LevelRepulsionForMinor}
Let $J$ be an $n \times n$ symmetric matrix satisfying conditions \eqref{eq:M_rigidity} and \eqref{eq:M_evectorl_infinity}.
Let $k \in [n]$, and let $J^{(k)}$ be the $(n-1) \times (n-1)$ principal submatrix of $J$ with row and column $k$ removed.
Let $\mu \in \left[2-n^{-2/3}\varphi_n^{3\rho},\,2+n^{-2/3}\varphi_n^{3\rho}\right]$ be an eigenvalue
of $J^{(k)}$.
If $J$ or $J^{(k)}$ satisfies \eqref{eq:LevelRepulsion}, then
there exists an eigenvalue $\lambda$ of $J$ such that
\begin{equation}  \label{eq:lambda-mu}
  0 \le \lambda-\mu \le  \frac{\varphi_n^C}{n}. 
\end{equation}
Consequently, if one of the matrices $J$ or
$J^{(k)}$ satisfies condition \eqref{eq:LevelRepulsion}, then the other one satisfies the same condition with a extra loss of $\frac{\varphi_n^C}{n}$.
\end{lem}
\begin{proof}
Note that $\mu$ is an eigenvalue of the matrix
$J-e_k e_k^\top J$ 
 as well since the $k$-th row of this matrix is $0$.
We will start with showing that there exists an eigenvalue $\lambda$ of $J$ satisfying \eqref{eq:lambda-mu}.
Let $G_{J}$ be the Green function of $J$. By Sylvester\textquoteright s
determinant identity, we have
\begin{align*}
0 & =\det\left(J-\mu-e_{k}e_{k}^{\top}J\right)\\
 & =\det\left(J-\mu\right)\det\left(I_{n}-e_{k}e_{k}^{\top}JG_{J}\left(\mu\right)\right)\\
 & =\det\left(J-\mu\right)\left(1-e_{k}^{\top}JG_{J}\left(\mu\right)e_{k}\right).
\end{align*}

If $\det\left(J-\mu\right)=0$, then we are done. Otherwise, $1-e_{k}^{\top}JG_{J}\left(\lambda\right)e_{k}=0$,
which can be rewritten as
\[
\sum_{\alpha}\frac{\lambda_{\a}}{\lambda_{\a}-\mu}\iprod{e_{k}}{u_{\a}}^{2}=1,
\]
where $\lambda_1 \ge \cdots \ge \lambda_m$ are the eigenvalues of $J$, and $u_1, \ldots, u_m$ are the corresponding unit eigenvectors.

For $\lambda_\alpha<0$, we have $0<\frac{\lambda_\alpha}{\lambda_\alpha-\mu}<\frac{2}{3}$ where the upper bound is due to $\lambda_\a >-3$ by
\eqref{eq:M_rigidity}.
Then,
\[
\sum_{\a,\lambda_{\a}<0}
\frac{\lambda_{\a}}{\lambda_{\a}-\mu}\iprod{e_{k}}{u_{\a}}^{2}
\le
\sum_{\a,\lambda_{\a}<0}
\frac{2}{3}\iprod{e_{k}}{u_{\a}}^{2}
\le \frac{2}{3}.
\] 
Hence,
\[
\sum_{\a,\lambda_{\a}>\mu}\frac{\lambda_{\a}}{\lambda_{\a}-\mu}\iprod{e_{k}}{u_{\a}}^{2}
\ge \sum_{\a,\lambda_{\a} \ge 0}\frac{\lambda_{\a}}{\lambda_{\a}-\mu}\iprod{e_{k}}{u_{\a}}^{2}
 \ge \frac{1}{3}
\]
as $\frac{\lambda_{\a}}{\lambda_{\a}-\mu} \le 0 $ for all $\lambda_\alpha \in [0, \mu)$.

Let $\beta$ be the largest positive integer so that $\lambda_{\beta}>\mu$.
Together with \eqref{eq:M_rigidity}, we have
\[
  2-n^{-2/3}\varphi_n^{3\rho} \le \mu \le \lambda_{\beta}
  \le \gamma_1+ n^{-2/3}\varphi_n^{3\rho}
  \le 2+n^{-2/3}\varphi_n^{3\rho}
\]
and hence
\[
  |2-\gamma_\b| \le |2-\lambda_\b| + |\lambda_\b - \gamma_\b|
  \le 2 n^{-2/3}\varphi_n^{3\rho}.
\]
With the estimate of $\gamma_\b$ in \eqref{eq:gammaalpha},
we conclude that
\begin{equation}
  \b \le \varphi_n^{C\rho}. \label{eq:EdgeBetaEstimate}
\end{equation}

Assume that $\beta>1$, and let $\a < \beta$.
If $J$ satisfies \eqref{eq:LevelRepulsion}, then
\[
\lambda_\a-\mu \ge \lambda_{\beta-1} - \lambda_{\beta} \ge n^{-2/3-\eps_{LR}}.
\]
On the other hand, assume that  $J^{(k)}$ satisfies \eqref{eq:LevelRepulsion}, and let  $\mu'$ be the smallest eigenvalue of $J^{(k)}$ which is greater than $\mu$. Due to the Cauchy interlacing theorem, we know that
$$
  \mu<\lambda_\b <\mu' <\lambda_\a.
$$
Then,
\[
\lambda_\a-\mu \ge \mu' - \mu \ge n^{-2/3-\eps_{LR}}.
\]
In both cases,  \eqref{eq:EdgeBetaEstimate}, \eqref{eq:M_evectorl_infinity} and \eqref{eq:M_rigidity} applied with $\alpha=1$  imply
\[
\sum_{\a<\beta}\frac{\lambda_{\a}}{\lambda_{\a}-\mu}\iprod{e_{k}}{u_{\a}}^{2}
\le\beta \frac{\lambda_{1}}{n^{-2/3-\eps_{LR}}} \max_{\alpha}\norm{u_{\alpha}}_{\infty}^2
=O\left(n^{-1/3+C\eps_{LR}}\right).
\]
If $\beta=1$, the inequality above is vacuous.
Thus, in both cases,
\[
\frac{\lambda_{\beta}}{\lambda_{\beta}-\mu}\iprod{e_{k}}{u_{\a}}^{2}\ge\frac{1}{3}+O\left(n^{-1/3+C\eps_{LR}}\right)
\] 

which in combination with \eqref{eq:M_rigidity}, \eqref{eq:M_evectorl_infinity}   leads to
\[
\frac{\varphi_n^C}{n}\ge \lambda_{\beta}-\mu >0
\] 
establishing \eqref{eq:lambda-mu}. Since  \eqref{eq:lambda-mu}  holds for all $\mu \in \left[2-n^{-2/3}\varphi_n^{3\rho},\,2+n^{-2/3}\varphi_n^{3\rho}\right]$, the second part of the lemma follows from  \eqref{eq:LevelRepulsion} for one of the  matrices $J$ or  $J^{(k)}$ and interlacing of their eigenvalues.
\end{proof}

Equipped with Lemma \ref{lem:LevelRepulsionForMinor}, we derive the desired result about the typical behavior of the principal submatrices.
We remind the reader that for convenience, we consider graphs with $n+2$ vertices.
\begin{thm} \label{thm:typical submatrices}
Let ${A}$ be the adjacency matrix of a $G(n+2,p)$ graph, and let
\[
  H= \frac{1}{\sqrt{p\left(1-p\right)(n+2)}}{A}
    - \sqrt{\frac{p(n+2)}{1-p}}\one\one^{\top}
    - \sqrt{\frac{p}{(1-p)(n+2)}}I_n,
\]
where $\one\in S^{n+1}$ is the vector such that every component equals $\frac{1}{\sqrt{n+2}}$. 
 Let $\scr A$ be the set of $(n+2) \times (n+2)$ symmetric matrices $H$ such that the matrix itself belongs to $\scr A_{(n+2,2)}$, all its principal $(n+1) \times (n+1)$ submatrices belong to $\scr A_{(n+1,3)}$,
and  all its principal  $n \times n$ submatrices belong to $\scr A_{(n,4)}$.

 Then
 \[
  \PP(H \in \scr A) \ge 1-n^{-\delta}
 \]
 for some $\delta=\delta(p,\,\rho,\,\eps_{LR})>0$.
 Moreover, for any $i,j \in [n]$,
 \[
   \EE \left| \ind_{\scr A_{(n,0)} } (H^{(i,j)})  - \ind_{\scr A} (H)   \right| \le n^{-1/3+ 2 \e_{LR}}.
 \]
\end{thm}
\begin{proof}
For \eqref{eq:IsotropicGreenFunction} and \eqref{eq:IsotropicDelocalization},
we use the probability estimate in \cite[Theorem 2.12, 2.16]{EJP3054}.
For \eqref{eq:M_rigidity} and \eqref{eq:M_evectorl_infinity}, we use the probability
estimate in \cite[Theorem 2.1, 2.2]{Erdos2012b}.
Combining them, we conclude that
\eqref{eq:IsotropicGreenFunction} -- \eqref{eq:IsotropicDelocalization}
hold for the matrix $H$ itself, as well as for all its $(n+1) \times (n+1)$
and $n \times n$ principal submatrices with probability at least $1-n^{-1}$.

In addition to it, \eqref{eq:LevelRepulsion} holds for $H$ with $k=2$
with probability at least $1-n^{-\delta}$. Then
Lemma \ref{lem:LevelRepulsionForMinor}, together with the properties
\eqref{eq:IsotropicGreenFunction} --  \eqref{eq:IsotropicDelocalization}
allow us to extend \eqref{eq:LevelRepulsion} with $k=3$ to all
its $(n+1) \times (n+1)$ principal minors. As these minors possess the same properties,
\eqref{eq:LevelRepulsion} further  extends with $k=4$ to all $n \times n$ principal minors.
Let us prove the second inequality. Denote by $\scr B$ the set of all $(n+2) \times (n+2)$ symmetric matrices satisfying conditions \eqref{eq:IsotropicGreenFunction} --  \eqref{eq:IsotropicDelocalization}.
Then
 \[
   \PP \left(  H^{(i,j)} \in \scr A_{(n,0)} \textrm{ and } H \notin \scr A \right)
   \le  \PP \left(  H^{(i,j)} \in \scr A_{(n,0)} \textrm{ and } H \notin \scr A  \textrm{ and } H \in \scr B\right)
   + \PP(H \notin \scr B)
   \le  n^{-1}
 \]
 since by Lemma \ref{lem:LevelRepulsionForMinor}, $\scr A_{(n,0)} \cap \scr A^c  \cap \scr B \subset \scr A_{(n,0)} \cap \scr A^c_{(n+2,2)}  \cap \scr B = \varnothing$.
 Also, notice that all the minors $H^{(i,j)}$ have the same distribution, so the value of $\theta$ is the same for all $i,j$.
 Hence,
 \[
   \PP \left(  H^{(i,j)} \notin \scr A_{(n,0)} \textrm{ and } H \in \scr A \right)
   \le  \PP \left(  H^{(i,j)} \notin \scr A_{(n,0)} \textrm{ and } H^{(i,j)}  \in \scr A_{(n,4)} \right)
   \le  n^{-1/3+ 2 \e_{LR}}
 \]
 by Lemma \ref{lem:change t}.
 The result follows.
\end{proof}

\subsection{\label{subsec:Typical B} Introduction of the shift }

In this section, we will derive the typical properties of
all $n \times n$ principal submatrices of $\t {A}$.
Recall that we denoted such submatrix by $B$, and
\begin{equation}\label{eq:B and M}
B=M+\sqrt{\frac{p\left(n+2\right)}{\left(1-p\right)}}ll^{\top}
\end{equation}
where $M$ is an $n \times n$ principal submatrix of $H$, and
$l=\left(\frac{1}{\sqrt{n+2}},\dots,\,\frac{1}{\sqrt{n+2}}\right)$
is almost a unit vector. We expect $B$ to behave close to $M$
 in a sense that its non-leading eigenvalues and eigenvectors possess similar properties.
The argument at this stage is deterministic.
We fix the matrix $M \in \scr{A}_{(n,0)}$ and treat $B$ as its rank one perturbation.

We start with showing that the non-leading edge eigenvalues
of $B$ are very close to that of $M$.

\begin{lem}
\label{lem: R1SSticking}
Let $M \in \scr A_{(n,0)}$ 
 be an $n \times n$ symmetric matrix with eigenvalues $\nu_1 \ge \cdots \ge \nu_n$, and let $B$ be as in \eqref{eq:B and M}.
Let $\mu_1 \ge \cdots \ge \mu_n$ be the eigenvalues of $B$.
If  $\beta$ is such that $\left|\nu_{\beta}-2\right|\le n^{-2/3}\varphi_n^{2\rho}$, then
\begin{equation}
 \left|\nu_{\b}-\mu_{\b+1}\right|\le n^{-1+C\eps_{LR}} \label{eq:stickness-1}
\end{equation}
for some universal constant $C>0$. Furthermore, $\mu_{\b+1}$ is an eigenvalue
of $M$ if and only if $\iprod l{v_\b}=0$,
where $v_\b$ is a unit eigenvector of $M$ corresponding to $\nu_\b$.
In the case $\mu_{\b+1}$ is not an eigenvalue of $M$, we have
\begin{equation}
\label{eq:stick2}
  \frac{\iprod l{v_{\b+1}}^{2}}{\nu_{\b}-\mu_{\b+1}}\ge1-o\left(1\right)
\end{equation}

\end{lem}
 We remark that \eqref{eq:stickness-1} is
  a simple case of \cite[Theorem 2.7]{KY13}, which deals with a deterministic finite rank shift.

\begin{proof}
Suppose that $\mu$ is an eigenvalue of $B$. By Sylvester's determinant
identity we have
\begin{align*}
0 & =\det\left(M-\mu I_{n}+\sqrt{\frac{p\left(n+2\right)}{1-p}}ll^{\top}\right)\\
 & =\det\left(M-\mu I_{n}\right)\det\left(I_{n}+G_{M}\left(\mu\right)\sqrt{\frac{p\left(n+2\right)}{1-p}}ll^{\top}\right)\\
 & =\det\left(M-\mu I_{n}\right)\left(1+l^{\top}G_{M}\left(\mu\right)\sqrt{\frac{p\left(n+2\right)}{1-p}}l\right),
\end{align*}
and $\left(1+l^{\top}G_{M}\left(\mu\right)\sqrt{\frac{p\left(n+2\right)}{1-p}}l\right) = 0$ if
\begin{equation}
\sum_{\a\in\left[n\right]}\frac{\iprod l{v_{\a}}^{2}}{\nu_{\a}-\mu}
=-\frac{1}{\sqrt{\frac{p\left(n+2\right)}{1-p}}}.\label{eq:R1S_NewEigenvalue}
\end{equation}
The matrix $B$ is a rank one positive semidefinite perturbation of $M$, so the eigenvalues of $M$ and $B$ are interlacing:
\begin{equation}
\mu_{1}\ge\nu_{1}\ge\mu_{2}\ge\dots\ge\mu_{n}\ge\nu_{n}.\label{eq:R1SInterlacing}
\end{equation}

For the leading eigenvalue, $\mu_1\ge\frac{1}{2}\sqrt{\frac{p(n+2)}{1-p}}$ due to the fact that $\norm{M}=O(1)$ by \eqref{eq:M_rigidity}.

Let $\b$ be such that $|\nu_{\b}-2|<n^{-2/3}\varphi_n^{2\rho}$. We consider two cases. First,  assume that $\iprod l{v_\a} \neq 0$ for $\a \in \{\b, \b+1\}$. Then $\mu_{\b+1}\notin \{\nu_{\b},  \nu_{\b+1}\}$, so  $\det\left(M-\mu_{\b+1} I_{n}\right) \neq 0$, and \eqref{eq:R1S_NewEigenvalue} holds.

We claim that
\begin{equation}
\sum_{\a\neq\b}\frac{\iprod l{v_{\a}}^{2}}{\nu_{\a}-E}\le-1+o\left(1\right)\label{eq:R1S_Evalue_OtherTerms}.
\end{equation}
for all $E\in\left(\nu_{\b+1},\,\nu_{\b}\right)$.

If the claim is proved, then, by \eqref{eq:R1S_NewEigenvalue},
\[
\frac{\iprod l{v_{\b+1}}^{2}}{\nu_{\b}-\mu_{\b+1}}
 =-\frac{1}{\sqrt{\frac{p\left(n+2\right)}{1-p}}}-\sum_{\a\neq\b}\frac{\iprod l{v_{\a}}^{2}}{\nu_{\a}-\mu_{\b+1}}
 \ge1-o\left(1\right)
\]
By \eqref{eq:IsotropicDelocalization},
we have $\iprod l{v_{\b+1}}^{2}<n^{\eps_{LR}-1}$, which allows to conclude
that
\[
 0 < \nu_{\b}-\mu_{\b+1}\le n^{2\eps_{LR}-1}
\]
as required.

Assume now that $\iprod l{v_\a}=0$ for some $\a\in \{\b,\b+1\}$.
Considering  an infinitesimally small perturbation $M^{(\eps)}=\sqrt{1-\eps^2}M + \eps G$ with a GOE matrix $G$,  we can guarantee that $\iprod l{v_\a} \neq 0$ a.s. In this case, the perturbed eigenvalue $\mu_{\beta+1}^{(\eps)}$ of $M^{(\eps)}$ satisfies the inequality above. Letting $\eps \to 0$ and using the stability of eigenvalues, we conclude that $\mu_{\b+1}=\nu_{\b}$ completing the proof of \eqref{eq:stickness-1}. This argument also shows that $\mu_{\b+1}$ is an eigenvalue of $M$ if and only if $\iprod l{v_\b}=0$.

It remains to verify (\ref{eq:R1S_Evalue_OtherTerms}). This will be done by comparing the
right hand side of (\ref{eq:R1S_Evalue_OtherTerms}) with
\[
\Re{\iprod l{G_{M}\left(E+\i\eta\right)l}}=\sum_{\a\in\left[n\right]}\frac{\nu_{\a}-E}{\left(\nu_{\a}-E\right)^{2}+\eta^{2}}\iprod l{v_{\a}}^{2}.
\]

Assume first that $\frac{1}{2}\nu_{\b}+\frac{1}{2}\nu_{\b+1}\le E\le\nu_{\b}$.
In view of \eqref{eq:LevelRepulsion},
\[
\nu_{\b+1}+\frac{1}{2}n^{-2/3-\eps_{LR}}<\frac{\nu_{\b+1}+\nu_{\b}}{2}<E<\nu_{\b}<\nu_{\b-1}-n^{-2/3-\eps_{LR}}.
\]

(we omit the last inequality if $\b=1$.) Hence, for $\a\neq\b$, we
have
\[
\left|E-\nu_{\a}\right|>\frac{1}{2}n^{-2/3-\eps_{LR}}=\frac{1}{2}\eta n^{\eps_{LR}}\label{eq:R1S_Etonu}
\]
 (recall that $\eta = n^{-2/3-2\eps_{LR}}$) and so
\[
\frac{1}{\nu_{\alpha}-E}=\left(1+O\left(n^{-2\eps_{LR}}\right)\right)\frac{\nu_{\alpha}-E}{\left(\nu_{\alpha}-E\right)^{2}+\eta^{2}}.
\]
Therefore,
\begin{align*}
\sum_{\a>\b}\frac{1}{\nu_{\alpha}-E}\iprod l{v_{\a}}^{2} & =\left(1+O\left(n^{-2\eps_{LR}}\right)\right)\sum_{\a>\b}\frac{\nu_{\a}-E}{\left(\nu_{\a}-E\right)^{2}+\eta^{2}}\iprod l{v_{\a}}^{2}\\
 & =\left(1+O\left(n^{-2\eps_{LR}}\right)\right)\left(\Re{\iprod l{G_{M}\left(E+\i\eta\right)l}}-\sum_{\a\le\b}\frac{\nu_{\a}-E}{\left(\nu_{\a}-E\right)^{2}+\eta^{2}}\iprod l{v_{\a}}^{2}\right).
\end{align*}
since all the summands have the same sign. Now we will evaluate the
two terms in the brackets.
The first one can be approximated using the local semicircular law, and the second one is negligible, because the sum consists of a few terms, and each term is small.
Indeed, using \eqref{eq:M_rigidity} and \eqref{eq:gammaalpha}, we have
\[\b\le\varphi_n^{C\rho}.\]
(The argument is the same as that for \eqref{eq:EdgeBetaEstimate}.)

With the trivial
bound $\left|\nu_{\a}-E\right|<2n^{-2/3}\varphi_n^{2\rho}$, we get
\[
\left|\sum_{\a\le\b}\frac{\nu_{\a}-E}{\left(\nu_{\a}-E\right)^{2}+\eta^{2}}\iprod l{v_{\a}}^{2}\right|\le\b\frac{n^{-2/3}\varphi_n^{3\rho}}{\eta^{2}}n^{-1+\eps_{LR}}\le n^{-1/3+6\eps_{LR}}
\]
if $n$ is sufficiently large. The isotropic  local semicircular law (\ref{eq:IsotropicGreenFunction}) yields
\begin{align*}
\Re{\iprod l{G_{M}\left(E+\i\eta\right)l}} & =\iprod ll \Re m_{sc}\left(E+\i\eta\right)+O\left(n^{-1/3+3\eps_{LR}}\right).
\end{align*}
Using the fact that $m_{sc}\left(z\right)=\frac{-z+\sqrt{z^{2}-4}}{2}$
with the branch cut at $\left[-2,\,2\right]$, for $|z-2|<s<1$
we have
\[
  m_{sc}\left(z\right)=-1+O(\sqrt{s}).
\] Thus,
\[
\Re{\iprod l{G_{M}\left(E+\i\eta\right)l}}=-1+O\left(n^{-1/3+3\eps_{LR}}\right)
\]
and we conclude that
\[
\sum_{\a>\b}\frac{1}{\nu_{\alpha}-E}\iprod l{v_{\a}}^{2}\le-1+o\left(1\right)\label{eq:R1S_ageb}
\]
for all $E\in\left(\frac{1}{2}\nu_{\b}+\frac{1}{2}\nu_{\b+1},\,\nu_{\b}\right)$.
Since $E\mapsto\sum_{\a>\b}\frac{1}{\nu_{\alpha}-E}\iprod l{v_{\a}}^{2}$
is increasing for $E>\nu_{\b+1}$, the inequality above
extends to all $E\in\left(\nu_{\b+1},\,\nu_{\b}\right)$. Together with
\[
\sum_{\a<\b}\frac{1}{\nu_{\alpha}-E}\iprod l{v_{\a}}^{2}\le\b\frac{1}{n^{-2/3-\eps_{LR}}}n^{-1+\eps_{LR}}=o\left(1\right)
\]
for $E\in\left(\nu_{\b+1},\,\nu_{\b}\right)$, we conclude that all $E\in\left(\nu_{\b+1},\,\nu_{\b}\right)$ satisfy
\[
\sum_{\a\neq\b}\frac{1}{\nu_{\alpha}-E}\iprod l{v_{\a}}^{2}\le-1+o\left(1\right),
\]
completing the proof of the lemma.
\end{proof}

Our next aim is comparing the Stieltjes transform of $B$ to that of the semicircular law. This will be done via the comparison of the former to the Stieltjes transform of $M$.

\begin{lem}  \label{lem:Stieltjes}
 Let $M \in \scr A_{(n,0)}$ be an $n \times n$ symmetric matrix, and let $B$ be as in \eqref{eq:B and M}.
 Then
 \[
\sup_{E:\,\left|E-2\right|\le\varphi_n^{2\rho}}\left|m_{B}\left(E+\i\eta\right)-m_{sc}\left(E+\i\eta\right)\right|\le n^{-1/3+C\eps_{LR}},
\]
where
\[
 m_{B}\left(z\right):=\frac{1}{n}\sum_{\a=1}^{n}\frac{1}{ \mu_{\a}-z}
\]
is the Stieltjes transform of $B$ and $\eta=n^{-2/3-2\eps_{LR}}$.
\end{lem}
\begin{proof}
Fix $E$ such that $\left|E-2\right|\le\varphi_n^{2\rho}$.
We estimate the real part
and imaginary of the Stieltjes transform part separately. Let us start with the real part.
\[
\Re{m_{B}\left(E+\i\eta\right)}=\frac{1}{n}\sum_{\a}\frac{\mu_{\a}-E}{\left(\mu_{\a}-E\right)^{2}+\eta^{2}}.
\]

 Let $\beta$ be the
smallest integer such that $\nu_{\b}<E-\eta$. Recall that we have
the interlacing property:
\[
E-\eta>\nu_{\b}\ge\mu_{\b+1}\ge\nu_{\beta+1}\ge\mu_{\b+2}\cdots\ge\mu_{n}\ge\nu_{n}.
\]
The function $x\rightarrow\frac{x}{x^{2}+\eta^{2}}$ is decreasing
when $\left|x\right|>\eta$. Based on this fact,
 we obtain
\[
\sum_{\a=\b}^{n-1}\frac{\nu_{\a}-E}{\left(\nu_{\a}-E\right)^{2}+\eta^{2}}\le\sum_{\a=\b+1}^{n}\frac{\mu_{\a}-E}{\left(\mu_{\a}-E\right)^{2}+\eta^{2}}\le\sum_{\a=\b+1}^{n}\frac{\nu_{\a}-E}{\left(\nu_{\a}-E\right)^{2}+\eta^{2}}.
\]
Furthermore, as $\frac{x}{x^{2}+\eta^{2}}$  lies in
$\left[-\frac{1}{2\eta},\frac{1}{2\eta}\right]$ for all $x \in \R$, we have
\[
\Re{m_{M}\left(E+\i\eta\right)}-\frac{\b}{n\eta}\le\Re{m_{B}\left(E+\i\eta\right)}\le\Re{m_{M}\left(E+\i\eta\right)}+\frac{\beta}{n\eta},
\]
and the bound for the real part follows.

For the imaginary part we have
\[
\Im{m_{B}\left(E+\i\eta\right)}=\frac{1}{n}\sum_{\a}\frac{\eta}{\left(\lambda_{B}-E\right)^{2}+\eta^{2}}.
\]
The function $x\rightarrow\frac{\eta}{x^{2}+\eta^{2}}$ is increasing
if $x<0$, hence
\[
\sum_{\a=\b+1}^{n-1}\frac{\eta}{\left(\nu_{\a}-E\right)^{2}+\eta^{2}}
\le\sum_{\a=\b+1}^{n}\frac{\eta}{\left(\mu_{\a}-E\right)^{2}+\eta^{2}}
\le\sum_{\a=\b}^{n}\frac{\eta}{\left(\nu_{\a}-E\right)^{2}+\eta^{2}}.
\]
Since $\frac{\eta}{x^{2}+\eta^{2}}\in\left[0,\,\frac{1}{\eta}\right]$
for all $x$, we conclude that
\[
\Im{m_{M}\left(E+\i\eta\right)-\frac{2\b}{n\eta}}\le\Im{m_{B}\left(E+\i\eta\right)}\ge\Im{m_{M}\left(E+\i\eta\right)}+\frac{2\b}{n\eta}.
\]
Similar to how we derive \eqref{eq:EdgeBetaEstimate},
using \eqref{eq:M_rigidity} and \eqref{eq:gammaalpha}, we have
\[\b\le\varphi_n^{C\rho}.\]

We conclude that
\[
\left|m_{M}\left(E+\i\eta\right)-m_{B}\left(E+\i\eta\right)\right|\le\varphi_n^{C\rho}n^{-1/3+2\eps_{LR}}.
\]
In view of  \eqref{eq:IsotropicGreenFunction},
\[
\left|m_{M}\left(E+\i\eta\right)-m_{sc}\left(E+\i\eta\right)\right|=\left|\frac{1}{n}\sum_{i}\iprod{e_{i}}{G\left(E+\i\eta\right)e_{i}}-m_{sc}\left(E+\i\eta\right)\right|\le3n^{-\frac{1}{3}+3\eps_{LR}}
\]
which in combination with the previous inequality finishes the proof.
\end{proof}

Next, we will derive the delocalization properties of edge eigenvectors of $B$.
\begin{lem}
\label{lem: R1S_EigenVector}
Let $M \in \scr A_{(n,0)}$ be an $n \times n$ symmetric matrix, and let $B$ be as in \eqref{eq:B and M}.
Let $\mu_1 \ge \cdots \ge \mu_n$ be the eigenvalues of $B$, and let $u_1, \ldots, u_n$ be the corresponding unit eigenvectors.
If  $\beta$ is such that $\left|\mu_{\beta+1}-2\right|\le n^{-2/3}\varphi_n^{2\rho}$, then
\begin{align}
\left|\iprod{u_{\b+1}}l\right| &\le n^{-1+C\eps_{LR}}.\label{eq:B_EigenvectorDotl-1}
\intertext{and}
\norm{u_{\b+1}}_{\infty} &\le\frac{n^{1/6+6\eps_{LR}}}{\sqrt{n}}.\label{eq:B_l_infinity-1}
\end{align}
\end{lem}
\begin{proof}
As pointed out in Lemma \ref{lem: R1SSticking}, $\mu_{\b+1}$ is an eigenvalue of $M$ if and only if $\iprod l{v_\b}=0$. In this case, we have $v_\b = u_{\b+1}$ so the statement follows trivially.

Now we assume $\mu_{\b+1}$ is not an eigenvalue of $M$, in which case, it satisfies \eqref{eq:R1S_NewEigenvalue}. Using this equality, one can directly check that
\[
 u=\sum_{\a\in\left[n\right]}\frac{\iprod l{v_{\a}}}{\nu_{\a}-\mu_{\b+1}} v_{\a}
\]
is an eigenvector  of $B$ corresponding to eigenvalue $\mu_{\b+1}$.

First, we provide a lower bound for $\norm u_{2}$. By Lemma
\ref{lem: R1SSticking}, we have $\frac{\iprod l{v_{\b+1}}^{2}}{\left|\nu_{\b}-\mu_{\b+1}\right|}\ge\frac{1}{2}$ and $ \left|\nu_{\b}-\mu_{\b+1}\right|\le n^{-1+C\eps_{LR}}$.
This allows to  bound the norm of $u$ by one of the coefficients:
\begin{equation} \label{eq: norm u}
\norm u_{2}^{2}\ge\frac{\iprod l{v_{\b+1}}^{2}}{\left|\nu_{\b}-\mu_{\b+1}\right|^{2}}
\ge\frac{1}{4}n^{1-C\eps_{LR}}.
\end{equation}

Recall that by (\ref{eq:R1S_NewEigenvalue}),
\[
 \iprod ul =  \sum_{\a\in\left[n\right]}\frac{\iprod l{v_{\a}}^{2}}{\nu_{\a}-\mu_{\b+1}}=-\frac{1}{\sqrt{\frac{p\left(n+2\right)}{1-p}}}.
\]
This yields
\begin{align*}
\left|\iprod{u_{\b+1}}l\right| &
=\frac{\left|\iprod ul \right|}{\norm u_{2}}\le n^{-1+C\eps_{LR}}
\end{align*}
if $n$ is sufficiently large.

Now we will estimate $\norm u_{\infty}=\max_{i\in\left[n\right]} \left| \sum_{\a\in\left[n\right]}\frac{\iprod l{v_{\a}}\iprod{e_{i}}{v_{\a}}}{\nu_{\a}-\mu_{\b+1}} \right|$.
We break the sum  isolating the main term:
\begin{align*}
\left|\iprod u{e_{i}}\right| & \le\left|\frac{\iprod l{v_{\b}}}{\nu_{\b}-\mu_{\b+1}}\right|\norm{v_{\b}}_{\infty}+\left|\sum_{\a\neq\b}\frac{\iprod l{v_{\a}}\iprod{e_{i}}{v_{\a}}}{\nu_{\a}-\mu_{\b+1}}\right|\\
 & \le\left|\frac{\iprod l{v_{\b}}}{\nu_{\b}-\mu_{\b+1}}\right|\norm{v_{\b}}_{\infty}+\sqrt{\sum_{\a\neq\b}\frac{\iprod l{v_{\a}}^{2}}{\left(\nu_{\a}-\mu_{\b+1}\right)^{2}}}\sqrt{\sum_{\a\neq\b}\iprod{e_{i}}{v_{\a}}^{2}}\\
 & \le\left|\frac{\iprod l{v_{\b}}}{\nu_{\b}-\mu_{\b+1}}\right|\norm{v_{\b}}_{\infty}+\sqrt{\sum_{\a\neq\b}\frac{\iprod l{v_{\a}}^{2}}{\left(\nu_{\a}-\mu_{\b+1}\right)^{2}}}.
\end{align*}

We will show below that
\begin{equation}
\sqrt{\sum_{\a\neq\b}\frac{\iprod l{v_{\a}}^{2}}{\left(\nu_{\a}-\mu_{\b+1}\right)^{2}}}\le n^{1/6+2\eps_{LR}}.\label{eq:R1S ImGreen}
\end{equation}
If this inequality holds, \eqref{eq: norm u} implies
\begin{align*}
\norm{u_{\b+1}}_{\infty} & =\frac{\norm u_{\infty}}{\norm u_{2}}\le\frac{\left|\frac{\iprod l{v_{\b}}}{\nu_{\b}-\mu_{\b+1}}\right|\norm{v_{\b}}_{\infty}}{\norm u_{2}}+\frac{n^{1/6+2\eps_{LR}}}{\norm u_{2}}\\
 & \le\frac{\left|\frac{\iprod l{v_{\b}}}{\nu_{\b}-\mu_{\b+1}}\right|\norm{v_{\b}}_{\infty}}{\left|\frac{\iprod l{v_{\b}}}{\nu_{\b}-\mu_{\b+1}}\right|}+4 n^{1/6-1/2+3\eps_{LR}}
 \le n^{-1/3 +4\eps_{LR}},
\end{align*}
where we used $\norm{v_{\b}}_{\infty}\le\frac{\varphi_n^{C}}{\sqrt{n}}$ from (\ref{eq:M_evectorl_infinity}) in the last inequality.
This completes the proof of the lemma modulus (\ref{eq:R1S ImGreen}).

In the rest of the proof, we focus on establishing (\ref{eq:R1S ImGreen})
by comparing $\sum_{\a\neq\b}\frac{\iprod l{v_{\a}}^{2}}{\left(\nu_{\a}-E\right)^{2}}$
with
\[
\frac{1}{\eta}\Im{\iprod l{G_{M}\left(E+\i\eta\right)l}}=\frac{1}{\eta}\Im{\sum_{\a\in\left[n\right]}\frac{\iprod l{v_{\a}}^{2}}{\nu_{\a}-E-\i\eta}}=\sum_{\a\in\left[n\right]}\frac{\iprod l{v_{\a}}^{2}}{\left(\nu_{\a}-E\right)^{2}+\eta^{2}}
\]
for any $E\in\left(\frac{\nu_{\b}+\nu_{\b+1}}{2},\,\nu_{\b}\right)$
which includes $\mu_{\b+1}$. The approach is basically the same as
in approximation of $\sum_{\a\neq\b}\frac{\iprod l{v_{\a}}^{2}}{v_{\a}-E}$
by $\Re{\iprod l{G\left(E+\i\eta\right)l}}$ in Lemma \ref{lem: R1SSticking}.
As in this lemma, we use $\left|\nu_{\a}-E\right|>\frac{1}{2}\eta n^{\eps_{LR}}$ for
$\a\neq\b$ to derive
\[
\frac{\eta}{\left(\nu_{\a}-E \right)^{2}}=\left(1+O\left(n^{-2\eps_{LR}}\right)\right)\frac{\eta}{\left(\nu_{\a}-E\right)^{2}+\eta^{2}}.
\]
Thus,
\begin{align*}
\sum_{\a\neq\b}\frac{\iprod l{v_{\a}}^{2}}{\left(\nu_{\a}-\mu_{\b+1}\right)^{2}} & =\left(1+O\left(n^{-2\eps_{LR}}\right)\right)\left[\frac{1}{\eta}\Im{\iprod l{G_{M}\left(E+\i\eta\right)l}}-\frac{\iprod l{v_{\b}}^{2}}{\left(\nu_{\b}-\mu_{\b+1}\right)^{2}+\eta^{2}}\right]\\
 & \le\left(1+O\left(n^{-2\eps_{LR}}\right)\right)\frac{1}{\eta}\Im{\iprod l{G_{M}\left(E+\i\eta\right)l}}.
\end{align*}
By (\ref{eq:IsotropicGreenFunction}) we have
\[
\Im{\iprod l{G\left(E+\i\eta\right)l}}=\Im{m_{sc}\left(E+\i\eta\right)}+O\left(n^{-1/3+3\eps_{LR}}\right).
\]

As $\left|E-2\right|<n^{-2/3}\varphi_n^{3\rho}$ and $\eta=n^{-2/3-2\eps_{LR}}$,
a direct estimate yields $\Im{m_{sc}\left(E+\i\eta\right)}=O\left(n^{-1/3}\varphi_n^{3\rho}\right)$.
Therefore,
\[
\sum_{\a\neq\b}\frac{\iprod l{v_{\a}}^{2}}{\left(\nu_{\a}-\mu_{\b+1}\right)^{2}}\le n^{1/3+4\eps_{LR}}
\]
proving (\ref{eq:R1S ImGreen}) and finishing the proof of the lemma.
\end{proof}

We have shown that if $M \in \scr A_{(n, 0)}$, then the matrix $B$ shares the spectral properties of $M$.
Let us summarize these properties.
\begin{definition}
\label{cond:Btypical}
Denote by $\scr T_{(n,k)}$ the set of
 $n \times n$ symmetric matrix $B$ with eigenvalues $\mu_1 \ge \cdots \ge \mu_n$ and unit eigenvectors $u_1, \ldots, u_n$  possessing the following properties.
\begin{itemize}
\item Eigenvalue properties:

\begin{itemize}
\item Local semicircular law:
\begin{equation}
\sup_{E:\,\left|E-2\right|\le\varphi_n^{2\rho}}\left|m_{B}\left(E+\i\eta\right)-m_{sc}\left(E+\i\eta\right)\right|\le n^{-1/3+C\eps_{LR}},\label{eq:B_localstatistic}
\end{equation}
where $m_{B}\left(z\right):=\frac{1}{n}\sum_{\a=1}^{n}\frac{1}{u_{\a}-z}$
is the Stieltjes transform of $B$ and $\eta=n^{-2/3-2\eps_{LR}}$.
\item Rigidity of the eigenvalues:
\begin{equation}
\forall\a=1,\dots,\,n-1\left|\mu_{\a+1}-\gamma_{\a}\right|\le
\varphi_n^{2C_{re}}\left[\min\left(\a,\,n-\a+1\right)\right]^{-1/3}n^{-2/3},
\label{eq:B_rigidity}
\end{equation}
\item Leading eigenvalue:
\begin{equation}
\mu_{1}\ge\frac{1}{2}\sqrt{\frac{p}{1-p}n}.\label{eq:B_LeadingEvalue}
\end{equation}
\end{itemize}
\item Edge eigenvector properties:
\begin{itemize}
\item Isotropic delocalization: \\
 for $\b$ such that $\left|\mu_{\b}-2\right|\le n^{-2/3}\varphi_n^{2\rho}$,
we have
\begin{equation}
\iprod{u_{\b}}l=O\left(n^{-1+c\eps_{LR}}\right).\label{eq:B_EvectorDotl}
\end{equation}
\item $\ell_{\infty}$ delocalization: \\
 for $\b$ such that $\left|\mu_{\b}-2\right|\le n^{-2/3}\varphi_n^{2\rho}$,
\begin{equation}
\norm{u_{\b}}_{\infty}\le\frac{n^{ 1/6+4\eps_{LR}}}{\sqrt{n}}.\label{eq:eq:B_EvectorLinifinity}
\end{equation}
\end{itemize}
\item Level repulsion at the edge:
$B \in \scr{LR} \left(n, \theta n^{-2/3-\eps_{LR}} -k \frac{\varphi^C_n}{n} \right)$, i.e., \\
for any two distinct eigenvalues $\nu,\,\nu'$ of $B$ in
$\left[2-n^{-2/3}\varphi_n^{3\rho},\,2+n^{-2/3}\varphi_n^{3\rho}\right]$,
we have
\begin{equation} \label{eq:B_LevelRepulsion}
\left|\nu-\nu'\right|>\theta n^{-2/3-\eps_{LR}} -k \frac{\varphi^C_n}{n}.
\end{equation}
\end{itemize}
\end{definition}
The matrices $B \in \scr T_{(n, 1)}$ will be called typical below.
In particular, we've shown that $M \in \scr A_{(n, 0)}$ implies $B \in \scr T_{(n, 1)}$.

Theorem  \ref{thm:typical submatrices} implies that probability close to $1$, the normalized adjacency matrix of a $G(n,p)$ graph is typical along with its principal submatrices. We will formulate it as a corollary.
\begin{cor} \label{cor:typical submatrices}
Let ${A}$ be the adjacency matrix of a $G(n+2,p)$ graph, and let
\[
  \t {A}= \frac{1}{\sqrt{p\left(1-p\right)(n+2)}}{ A}
  - \sqrt{\frac{p}{(1-p)(n+2)}}I_{n+2}
\]
and
\[
  H = \t {A} - \sqrt{\frac{p(n+2)}{1-p}}\one \one^\top.
\]
Let $\scr{T}$ be the set of all matrices $\t {A}$ such that $H \in \scr{A}$.
Then
 \[
  \PP( \t { A} \in \scr T) \ge 1-n^{-\delta}
 \]
 for some $\delta=\delta(p, \rho, \eps_{LR})>0$.
 Moreover, for any $i,j \in [n]$,
 \[
   \EE \left|
      \ind_{\scr T_{(n,1)} } (\t { A}^{(i,j)})
      - \ind_{\scr T} (\t { A})
    \right| \le n^{-1/3+ 2 \e_{LR}}.
 \]

\end{cor}
\begin{proof}
Except for \eqref{eq:B_rigidity} and \eqref{eq:B_LeadingEvalue}, these  conditions have been derived from the corresponding conditions on $H$ above. Condition \eqref{eq:B_rigidity} follows from the interlacing of the eigenvalues of $\tilde{A}_p$ and its principal submatrices.
 Finally, \eqref{eq:B_LeadingEvalue},
follows from (\ref{eq:M_rigidity}) for $\alpha=1$ since
\[
\mu_{1}\ge\iprod l{Bl}\ge\sqrt{\frac{p\left(n+2\right)}{1-p}}\norm l_{2}^{4}-\lambda_1(M) \norm l_{2}^{2}\ge\frac{1}{2}\sqrt{\frac{p}{1-p}n}.
\]
Both probability estimates follow now from Theorem \ref{thm:typical submatrices}.
\end{proof}

\subsection{Concentration of $w_{i}^{\top}G\left(E\right)w_{j}-d_{ij}+E$}

In this section, we fix an $n \times n$ matrix $B\in \scr T_{(n,1)}$. 
 Let $E$ be a constant such that $\left|E-2\right|\le n^{-2/3}\varphi_n^{2\rho}$.
Let $\left\{ \mu_{\a}\right\} _{\a=1}^{n}$ be eigenvalues of $B$ arranged
in the non-increasing order and let $\left\{ u_{\a}\right\} _{\a=1}^{n}$
be the corresponding unit eigenvectors.
Let $G\left(E\right)=\sum_{\a}\frac{1}{\mu_{\a}-E}u_{\a}u_{\a}^{\top}$ be the Green function of $B$.

Denote by $\a_{E}$ the integer such that
\[
\left|\mu_{\a_{E}}-E\right|=\min_{\a}\left|\mu_{\a}-E\right|.
\]
In this section we will prove the following lemma:
\begin{lem}
\label{lem:wGw-1}
Let $B\in \scr T_{(n,1)}$.
With probability greater than $1-\exp\left(-c\left(p\right)\varphi_n\right)$
($\varphi_n:=\left(\log n\right)^{\log\log n}$) in $w_{1}$
and $w_{2}$, we have
\begin{equation}
\forall i,\,j\in\left\{ 1,\,2\right\} \;w_{i}^{\top}G\left(E\right)w_{j}=-\left(1+O\left(n^{-2\eps_{LR}}\right)\right)\delta_{ij}+\frac{\iprod{w_{i}}{u_{\a_{E}}}\iprod{w_{j}}{u_{\a_{E}}}}{\mu_{\a_{E}}-E}+O\left(n^{-1/3+C\eps_{LR}}\right)\label{eq: wGw estimate-1-1}
\end{equation}
for all $E\in\left[2-n^{-2/3}\varphi_n^{2\rho},\,2+n^{-2/3}\varphi_n^{2\rho}\right]$
and $\a_{E}\in\left[n\right]$ is the integer so that $\left|\mu_{\a_{E}}-E\right|\le\min_{\a\in\left[n\right]}\left|\mu_{\a}-E\right|$.
\end{lem}

By level repulsion (\ref{eq:B_LevelRepulsion}), 
we have
\begin{equation}
\left|\mu_{\a}-E\right|>\frac{1}{8}n^{-2/3-\eps_{LR}}\label{eq:EdgeElocation}
\end{equation}
 for $\a\neq\a_{E}$.
 Decompose $G$ to separate the main term:
\[
G\left(E\right)=\sum_{\a \in [n]}\frac{1}{\mu_{\a}-E}u_{\a}u_{\a}^{\top}=\sum_{\a\neq\a_{E}}\frac{1}{\mu_{\a}-E}u_{\a}u_{\a}^{\top}+\frac{1}{\mu_{\a_{E}}-E}u_{\a_{E}}u_{\a_{E}}^{\top}:=L\left(E\right)+\frac{1}{\mu_{\a_{E}}-E}u_{\a_{E}}u_{\a_{E}}^{\top}.
\]
For $i=1,\,2$, we express $w_{i}$ as
\[
w_{i}=\t w_{i}+\sqrt{\frac{p}{1-p}}l,
\]
where $\t w_{i}$ has i.i.d components with the same  distribution
as in (\ref{eq:H_EntryDistribution-1}). In particular, one can treat
$\sqrt{n+2}\t w_{i}$ as an isotropic subgaussian vector whose entries
have $\psi_{2}$-norms bounded by $K\left(p\right)$.

Our goal is to show that $w_{i}^{\top}L\left(E\right)w_{j}$ is concentrated about $-\delta_{i,j}$.
To achieve that, we represent it as
\begin{equation}  \label{eq:decomposition w}
w_{i}^{\top}L\left(E\right)w_{j}=\t w_{i}^{\top}L\left(E\right)\t w_{j}+\sqrt{\frac{p}{1-p}}l^{\top}L\left(E\right)\t w_{j}+\sqrt{\frac{p}{1-p}}\t w_{i}^{\top}L\left(E\right)l+\frac{p}{1-p}l^{\top}L\left(E\right)l
\end{equation}
and estimate each summand separately. We start with the bilinear term.
\begin{lem}
\label{lem:wGw}
Fix an $n \times n$ matrix $B \in \scr{T}_{(n,1)}$. 
With probability greater than $1-\exp\left(-c\left(p\right)\varphi_n\right)$
($\varphi_n:=\left(\log n\right)^{\log\log n}$) in $w_{1}$
and $w_{2}$, we have
\begin{equation}
\t w_{i}^{\top}L\left(E\right)\t w_{j}=-\left(1+O\left(n^{-2\eps_{LR}}\right)\right)\delta_{ij}+O\left(n^{-1/3+C\eps_{LR}}\right)\label{eq: wLw estimate}
\end{equation}
for $E\in\left[2-n^{-2/3}\varphi_n^{2\rho},\,2+n^{-2/3}\varphi_n^{2\rho}\right]$.
Here, $O\left(n^{-2\eps_{LR}}\right)$ and $O\left(n^{-1/3+C\eps_{LR}}\right)$ mean some deterministic functions of $n$ with the prescribed asymptotic, and $c\left(p\right)$ is a constant that depends only on $p$.
\end{lem}
\begin{proof}[Proof of Lemma \ref{lem:wGw}]
Fix $E\in\left[2-n^{-2/3}\varphi_n^{2\rho},\,2+n^{-2/3}\varphi_n^{2\rho}\right]$.
We will first estimate the expectation of $\t w_{1}^{\top}L\left(E\right)\t w_{1}$
and then use the Hanson-Wright inequality to
derive the concentration.

First, we will estimate the expectation.

Since $\EE_{\t w_{1}, \t w_{2}}\t w_{1}^{\top}L\left(E\right)\t w_{2}=0$ by independence of $\t w_{1}$ and $\t w_{2}$, and since
$\EE_{\t w_{2}}\t w_{2}^{\top}L\left(E\right)\t w_{2}=\EE_{\t w_{1}}\t w_{1}^{\top}L\left(E\right)\t w_{1}$,
we have to evaluate only the last quantity.
Using the fact that $\t w_{1}$ has independent entries with mean $0$
and variance $\frac{1}{n+2}$, we obtain
\[
\EE_{\t w_{1}}\t w_{1}^{\top}L\left(E\right)\t w_{1}=\EE_{\t{w_{1}}}\sum_{\a\neq\a_{E}}\frac{1}{\mu_{\a}-E}\iprod{u_{\a}}{\t w_{1}}^{2}=\sum_{\a\neq\a_{E}}\frac{1}{\mu_{\a}-E}\frac{\sum_{i\in\left[n\right]}u_{\a}^{2}\left(i\right)}{n+2}=\frac{1}{n+2}\sum_{\a\neq\a_{E}}\frac{1}{\mu_{\a}-E}.
\]
Recall that for all $\a\in\left[n-1\right]$, we have rigidity of
eigenvalues (\ref{eq:B_rigidity}):
\[
\left|\mu_{\a+1}-\gamma_{\a}\right|\le2\varphi_n^{A_{sls}}\left[\min\left(\a,\,n-\a+1\right)\right]^{-1/3}n^{-2/3}.
\]
Hence, $\left|\left\{ \a\,:\,\mu_{\a}>E,\,\&\,\a\neq\a_{E}\right\} \right|\le\varphi_n^{C\rho}$,
and
\begin{equation}  \label{eq:GreaterthanE}
\sum_{\alpha:\,\mu_{\a}>E\,\&\a\neq\a_{E}}\frac{1}{\mu_{\a}-E}
 \le\left|\left\{ \a\,:\,\mu_{\a}>E,\,\&\,\a\neq\a_{E}\right\} \right| \cdot  \frac{1}{4}n^{2/3+\eps_{LR}}
  \le n^{2/3+2\eps_{LR}}
\end{equation}
We write
\[
\frac{1}{\mu_{\a}-E}  =\left(1+\frac{\eta^{2}}{\left(\mu_{\a}-E\right)^{2}}\right)\frac{\mu_{\a}-E}{\left(\mu_{\a}-E\right)^{2}+\eta^{2}},
\]
and set $\eta:=n^{-2/3-2\eps_{LR}}$.
With this choice of $\eta$,  we have $\left|\mu_{\a}-E\right|>\frac{1}{4}n^{\eps_{LR}}\eta$ from
(\ref{eq:EdgeElocation}), and so $\left(1+\frac{\eta^{2}}{\left(\mu_{\a}-E\right)^{2}}\right)=1+O\left(n^{-2\eps_{LR}}\right)$.
Therefore,
\begin{align}
 \frac{1}{n}\sum_{\a:\,\mu_{\a}<E\,\&\,\a\neq\a_{E}}\frac{1}{\mu_{\a}-E}
= & \left(1+O\left(n^{-2\eps_{LR}}\right)\right)\sum_{\a:\,\lambda_{\a}<E\,\&\,\a\neq\a_{E}}\frac{1}{n}\frac{\mu_{\a}-E}{\left(\mu_{\a}-E\right)^{2}+\eta^{2}} \notag \\
= & \left(1+O\left(n^{-2\eps_{LR}}\right)\right) \left[\Re{m_{B}\left(E+i\eta\right)}-\frac{1}{n}\sum_{\alpha:\,\mu_{\a}>E\,{\rm or}\,\a=\a_{E}}\frac{\mu_{\a}-E}{\left(\mu_{\a}-E\right)^{2}+\eta^{2}} \right]\notag \\
= & \left(1+O\left(n^{-2\eps_{LR}}\right)\right)\Re{m_{B}\left(E+i\eta\right)}+O\left(n^{-1/3+3\eps_{LR}}\right),\label{eq:ApproximatedStieljes}
\end{align}
where the last equality relies on \eqref{eq:GreaterthanE}.
Combining (\ref{eq:GreaterthanE}) and (\ref{eq:ApproximatedStieljes}),
we get
\[
\frac{1}{n}\sum_{\a\neq\a_{E}}\frac{1}{\mu_{\a}-E}=\left(1+O\left(n^{-2\eps_{LR}}\right)\right)\Re{m_{B}\left(E+i\eta\right)}+O\left(n^{-1/3+3\eps_{LR}}\right).
\]
We have $\Re{m_{B}\left(E+i\eta\right)}=\Re{m_{sc}\left(E+\i\eta\right)}+O\left(n^{-1/3+C\eps_{LR}}\right)=-1+O\left(n^{-1/3+C\eps_{LR}}\right)$
by (\ref{eq:B_localstatistic}).
Thus, if $\eps_{LR}$ is small enough, then
\[
\frac{1}{n}\sum_{\a\neq\a_{E}}\frac{1}{\mu_{\a}-E}=-1+O\left(n^{-C\eps_{LR}}\right).
\]
We conclude that
\[
\EE_{\t w_{1}}\t w_{1}^{\top}L\left(E\right)\t w_{1}=-1+O\left(n^{-C\eps_{LR}}\right).
\]

Now we are ready to derive concentration via
Hanson-Wright inequality \cite{Rudelson2013} by the second author and Vershynin.
\begin{thm}
    \cite{Rudelson2013}\label{thm:Hanson-Wright} Let $X=\left(X_{1},\dots,X_{n}\right)\in\R^{n}$
    be a random vector with independent components $X_{i}$ with satisfy
    $\EE X_{i}=0$, $ $and $\norm{X_{i}}_{\psi_{2}}\le K$. Let $A$
    be an $n\times n$ matrix. Then, for every $t\ge0$,

    \begin{align}
    \PP\left(\left|X^{\top}AX-\EE X^{\top}AX\right|>t\right) & \le2\exp\left(-c\min\left(\frac{t^{2}}{K^{4}\norm A_{HS}^{2}},\,\frac{t}{K^{2}\norm A}\right)\right)\label{eq:Hanson-Wright}
    \end{align}
\end{thm}

To this end, we need to estimate the operator norm
and Hilbert Schmidt norm of $L\left(E\right)$. The operator norm
can be estimated directly:
\[
\norm{L\left(E\right)}\le\max_{\a\neq\a_{E}}\frac{1}{\mu_{\a}-E}\le\frac{1}{4}n^{2/3+\eps_{LR}}.
\]
For the Hilbert Schmidt norm,
 a derivation similar to (\ref{eq:ApproximatedStieljes}) yields
\begin{align}
 \norm{L\left(E\right)}_{HS}^{2}=
  & \sum_{\a\neq\a_{E}}\frac{1}{\left(\mu_{\a}-E\right)^{2}}\nonumber \\
= & \left(1+o\left(1\right)\right)\sum_{\a\neq\a_{E}}\frac{1}{\left(\mu_{\a}-E\right)^{2}+\eta^{2}}=\left(1+o\left(1\right)\right)\frac{n}{\eta}\sum_{\a\neq\a_{E}}\frac{\eta}{n}\frac{1}{\left(\mu_{\a}-E\right)^{2}+\eta^{2}}\nonumber \\
= & \left(1+o\left(1\right)\right)\frac{n}{\eta} \left[\Im{m_{B}\left(E+\i\eta\right)}-\frac{\eta}{n}\frac{1}{\left(\mu_{\a_{E}}-E\right)^{2}+\eta^{2}} \right]\nonumber \\
= & \left(1+o\left(1\right)\right)\frac{n}{\eta}\left(\Im{m_{sc}\left(E+\i\eta\right)}+O\left(n^{-1/3+C\eps_{LR}}\right)-\frac{\eta}{n}\frac{1}{\left(\mu_{\a_{E}}-E\right)^{2}+\eta^{2}}\right), \label{eq:CompareToImStieltjesTransform}
\end{align}
where we used $\left|m_{sc}\left(E+\i\eta\right)-m\left(E+\i\eta\right)\right|\le O\left(n^{-1/3+C\eps_{LR}}\right)$
from (\ref{eq:B_localstatistic}). A direct computation shows
that $\Im\left(m_{sc}\left(E+\i\eta\right)\right)=O\left(n^{-1/3+C\eps_{LR}}\right)$
and $\frac{\eta}{n}\frac{1}{\left(\mu_{\a_{E}}-E\right)^{2}+\eta^{2}}=O\left(\frac{1}{n\eta}\right)=O\left(n^{-1/3+2\eps_{LR}}\right)$.
Hence,
\[
\norm{L\left(E\right)}_{HS}^{2}=\sum_{\a\neq\a_{E}}\frac{1}{\left(\mu_{\a}-E\right)^{2}}=\left(1+o\left(1\right)\right)\frac{n}{\eta}O\left(n^{-1/3+C\eps_{LR}}\right)=O\left(n^{4/3+C\eps_{LR}}\right).
\]

One can easily show that $\norm{\sqrt{n+2}\t w_{1}\left(i\right)}_{\psi_{2}}\le C\sqrt{\frac{1-p}{p}}$.
An application Hanson-Wright inequality with $X=\sqrt{n+2}\t w_{1}$
and $A=L\left(E\right)$ yields
\[
\PP\left(\left|\t w_{1}^{\top}L\left(E\right)\t w_{1}-\EE_{w_{1}}\t w_{1}^{\top}L\left(E\right)\t w_{1}\right|\ge\frac{t}{n+2}\right)\le2\exp\left(-c\left(p\right)\frac{t}{n^{2/3+C\eps_{LR}}}\right)
\]
for any $t>1$.
Taking $t=n^{2/3+2C\eps_{LR}}$, we get
\[
 \t w_{1}^{\top}L\left(E\right)\t w_{1}=
 \underbrace{-1+O\left(n^{-2\eps_{LR}}\right)}
 _{\EE\t w_{1}^{\top}L\left(E\right)\t w_{1}}+O\left(n^{-1/3+ 2C\eps_{LR}}\right)
\]
with probability at least $1-\exp\left(-c\left(p\right)\varphi_n\right)$.
(Recall that $\varphi_n =\log n^{\log\log n}$. )

Notice that, the same estimate works for $\t w_{2}$ and $\t w_{1}+\t w_{2}$
as well: with probability at least $1-\exp\left(-c\left(p\right)\varphi_n\right)$,
\begin{align*}
\left(\t w_{1}+\t w_{2}\right)^{\top}L\left(E\right)\left(\t w_{1}+\t w_{2}\right) &
=\underbrace{\EE\t w_{1}^{\top}L\left(E\right)\t w_{1}
+\EE\t w_{2}^{\top}L\left(E\right)\t w_{2}}
_{\EE\left(\t w_{1}+\t w_{2}\right)^{\top}L\left(E\right)
\left(\t w_{1}+\t w_{2}\right)}+O\left(n^{-1/3+2C\eps_{LR}}\right).
\end{align*}
Therefore,
by the linearity,  adjusting the constant
$C$ appropriately
we have
\[
\t w_{1}^{\top}L\left(E\right)\t w_{2}=O\left(n^{-1/3+C\eps_{LR}}\right),
\]
with probability at least
$1-\exp\left(-c\left(p\right)\varphi_n\right)$, thus obtaining (\ref{eq: wLw estimate}) for a fixed $E$.

To extend this to all $E\in\left[2-n^{-2/3}\varphi_n^{2\rho},\,2+n^{-2/3}\varphi_n^{2\rho}\right]$, we will use a net argument.
Let $\scr N$ be a $\kappa$-net in $\left[2-n^{-2/3}\varphi_n^{2\rho},\,2+n^{-2/3}\varphi_n^{2\rho}\right]$
with $\kappa=n^{-100}$ and assume that (\ref{eq: wLw estimate}) holds
for all $E\in\scr N$. Since $\left|\scr N\right|$ is polynomial
in $n$, this event has probability bounded by $\exp\left(-c\left(p\right)\varphi_n\right)$.

Recall that the coordinates of $\sqrt{n+2}\t w_{i}$ are independent, centered, subgaussian random variables with $\norm{\sqrt{n+2}\t w_{1}\left(k\right)}_{\psi_{2}}\le C\sqrt{\frac{1-p}{p}}$.
By Hoeffding's inequality,
\[
\sqrt{n+2}\iprod{\t w_{i}}{u_{\a}}=\sum_{k=1}^{n}\sqrt{n+2}\t w_{i}\left(k\right)u_{\a}\left(k\right)
\]
is also subgaussian since
$\norm{u_{\a}}_{2}=1$.
Similarly, $(n+2)\norm{\t w_i}_2^2$, being a sum of subexponential random variables, satisfies Bernstein's inequality.
Together with  a union bound, these two facts imply
\[
\PP\left(\exists\a\in\left[n\right],i\in\left\{ 1,\,2\right\} \,\left|\iprod{\t w_{i}}{u_{\a}}\right|\ge\frac{\varphi_n}{\sqrt{n+2}}\,\&\,\norm{w_{i}}_{2}\le\varphi_n\right)\le\exp\left(-c\left(p\right)n\right).
\]
Assume that these two events occur in addition to the assumption that (\ref{eq: wLw estimate}) holds
for all $E\in\scr N$ which we already made. Let $E\in\left[2-n^{-2/3}\varphi_n^{2\rho},\,2+n^{-2/3}\varphi_n^{2\rho}\right]$, and
choose $E'\in\scr N$ such
that $\left|E-E'\right|<\kappa$. Suppose that $\a_{E}\neq\a_{E'}$, then
\begin{align*}
 & \left|\t w_{i}^{\top}
 L\left(E\right)\t w_{j}-\t w_{i}^{\top}L\left(E'\right)\t w_{j}\right|\\
\le & \norm{\t w_{i}}_{2}\norm{\t w_{j}}_{2}
\sum_{\a\neq\a_{E},\,\a_{E'}}
\left|\frac{1}{\mu_{\a}-E}-\frac{1}{\mu_{\a}-E'}\right|+
\left|\frac{\iprod{\t w_{i}}{u_{\a_{E'}}}\iprod{\t w_{j}}{u_{\a_{E'}}}}{\mu_{\a_{E'}}-E}\right|+\left|\frac{\iprod{\t w_{ i}}{u_{\a_{E}}}\iprod{\t w_{j}}{u_{\a_{E}}}}{\mu_{\a_{E}}-E'}\right|\\
\le & \norm{\t w_{i}}_{2}\norm{\t w_{j}}_{2}
\sum_{\a\neq\a_{E},\,\a_{E'}}\frac{4\kappa}{\eta^{2}}+
\left|\frac{\iprod{\t w_{i}}{u_{\a_{E'}}}\iprod{\t w_{ j}}{u_{\a_{E'}}}}{\mu_{\a_{E'}}-E}\right|+\left|\frac{\iprod{\t w_{i}}{u_{\a_{E}}}\iprod{\t w_{j}}{u_{\a_{E}}}}{\mu_{\a_{E}}-E'}\right|\\
\le & \norm{\t w_{i}}_{2}\norm{\t w_{j}}_{2}\frac{4n}{\eta^{2}}\kappa+\left|\frac{\iprod{\t w_{i}}{u_{\a_{E'}}}\iprod{\t w_{j}}{u_{\a_{E'}}}}{\mu_{\a_{E'}}-E}\right|+\left|\frac{\iprod{\t w_{i}}{u_{\a_{E}}}\iprod{\t w_{j}}{u_{\a_{E}}}}{\mu_{\a_{E}}-E'}\right|
\end{align*}
Since $\a_{E}\neq\a_{E'}$, we have $\min\left\{ \left|\mu_{\a_{E'}}-E\right|,\,\left|\mu_{\a_{E}}-E'\right|\right\}
\ge\frac{1}{8}n^{-2/3-\eps_{LR}}$.
Together with $\left|\iprod{\t w_{i}}{u_{\a}}\right|\le\frac{\varphi_n}{\sqrt{n+2}}$,
this yields
\[
\left|\frac{\iprod{\t w_{i}}{u_{\a_{E'}}}\iprod{\t w_{j}}{u_{\a_{E'}}}}{\mu_{\a_{E'}}-E}\right|+\left|\frac{\iprod{\t w_{i}}{u_{\a_{E}}}\iprod{\t w_{j}}{u_{\a_{E}}}}{\mu_{\a_{E}}-E'}\right|=O\left(n^{-1/3+2\eps_{LR}}\right).
\]

Thus,
\begin{align*}
\left|\t w_{i}^{\top}L\left(E\right)\t w_{j}-\t w_{i}^{\top}L\left(E'\right)\t w_{j}\right| & \le\norm{\t w_{i}}_{2}\norm{\t w_{j}}_{2}\frac{4n}{\eta^{2}}\kappa+O\left(n^{-1/3+2\eps_{LR}}\right)
\end{align*}

As $\kappa=n^{-100},$ the difference is bounded by $O\left(n^{-1/3+2\eps_{LR}}\right)$.
The same bound holds for the case $\a_{E}=\a_{E'}$, and the proof is simpler, since the last two terms do not appear.
Therefore,  (\ref{eq: wLw estimate}) holds for $E$ as well
if constant $C$ is appropriately adjusted.
\end{proof}

Next,
we bound the linear and constant terms in \eqref{eq:decomposition w}.
\begin{lem}
\label{prop:Negligiblel}
Fix an $n \times n$ matrix $B \in \scr T_{(n,1)}$. 
 With probability greater than $1-\exp\left(c\left(p\right)\varphi_n^{C}n\right)$,
for any $E$ such that $\left|E-2\right|\le n^{-2/3}\varphi_n^{2\rho}$,
\begin{align}
l^{\top}L\left(E\right)l=O\left(n^{-1/3+C\eps_{LR}}\right),\,{\rm and}\,\t w_{1}^{\top}L\left(E\right)l=O\left(n^{-1/3+C\eps_{LR}}\right).\label{eq: wLw neg}
\end{align}
Here, $c\left(p\right)$ is a constant that depends only on $p$.
\end{lem}

\begin{proof}
Applcation of Hoeffding's inequality to $\iprod{\t w_{i}}{u_{\a}}$ yields
\[
\PP\left(\iprod{\t w_{i}}{u_{\a}}^{2}\ge\frac{\varphi_n}{n+2}\right)\le\exp\left(-c\left(p\right)\varphi_n\right),
\]
and so
\[
\max_{\a,\,i}\iprod{\t w_{i}}{u_{\a}}^{2}\le\frac{\varphi_n}{n}
\]
with probability greater than $1-\exp\left(-c\left(p\right)\varphi_n\right)$.
In view of this inequality and the fact that $\left(\sum_{\a\neq1}\iprod l{u_{\a}}^{2}\right)^{\frac{1}{2}}=\left|P_{u_{1}^{\perp}}l\right|=O\left(n^{-1/2+c \eps_{LR}}\right)$,
\begin{align*}
\left|\sum_{\a\ne1,\,\a_{E}}\frac{\iprod{\t w_{i}}{u_{\a}}\iprod l{u_{\a}}}{\mu_{\a}-E}\right| & \le\left(\sum_{\a\neq1,\,\a_{E}}\iprod l{u_{\a}}^{2}\right)^{\frac{1}{2}}\left(\sum_{\a\neq1,\,\a_{E}}\frac{\iprod{\t w_{i}}{u_{\a}}^{2}}{\left(\mu_{\a}-E\right)^{2}}\right)^{\frac{1}{2}}\\
 & =O\left(n^{-1+c' \eps_{LR}} \right)\sqrt{\sum_{\a\neq1,\,\a_{E}}\frac{1}{\left(\mu_{\a}-E\right)^{2}}}.
\end{align*}
Again, one can approximate $\sum_{\a\neq1,\,\a_{E}}\frac{1}{\left(\mu_{\a}-E\right)^{2}}$
by $\frac{n}{\eta}\Im{m_{sc}\left(E+\i\eta\right)}$ as  before
and obtain
\[
\sum_{\a\neq1,\,\a_{E}}\frac{1}{\left(\mu_{\a}-E\right)^{2}}=O\left(n^{4/3+C\eps_{LR}}\right).
\]
This shows that
\[
\left|\sum_{\a\ne1,\,\a_{E}}\frac{\iprod{\t w_{i}}{u_{\a}}\iprod l{u_{\a}}}{\mu_{\a}-E}\right|=O\left(n^{-1/3+C\eps_{LR}}\right).
\]
with probability greater than $1-\exp\left(-c\left(p\right)\varphi_n\right)$.

Furthermore, recall that by (\ref{eq:B_LeadingEvalue}), $\mu_{1}\ge\frac{1}{2}\sqrt{\frac{p\left(n+2\right)}{1-p}}$.
Thus $\left|\frac{\iprod{\t w_{i}}{u_{1}}\iprod l{u_{1}}}{\mu_{1}-E}\right|=o\left(\frac{1}{\sqrt{pn}}\right)$, and
\[
|l^{\top} L\left(E\right)l|=
\left|\sum_{\a\ne\a_{E}}\frac{\iprod l{u_{\a}}^{2}}{\mu_{\a}-E}\right|  \le\left(\frac{1}{4}n^{2/3+\eps_{LR}}\sum_{\a\ne1,\,\a_{E}}\iprod l{\t u_{\a}}^{2}\right)+\frac{1}{\mu_{1}-E}
  \le n^{-1/3+C\eps_{LR}}.
\]

Again, this result can extend easily for all $E\in\left[2-n^{-2/3}\varphi_n^{2\rho},\,2+n^{-2/3}\varphi_n^{2\rho}\right]$
by a net argument. We omit the proof here since it is the same as
 the net argument in Lemma \ref{lem:wGw}.
\end{proof}

Combining  Lemmas \ref{lem:wGw} and \ref{prop:Negligiblel},
we obtain Lemma \ref{lem:wGw-1}.

\subsection{Estimate of $s\left(\lambda\right)$}
Recall that in Corollary \ref{cor:typical submatrices}, we  denoted by
$\scr T$ be the set of $(n+2) \times (n+2)$ symmetric matrices all whose $n \times n$ principal submatrices are typical in a sense that they satisfy the conditions in $\scr T_{(n,5)}$.
Suppose that $\lambda_\alpha$ is an eigenvalue of $\t {A}$ and $v_\alpha \in\R^{n+2}$
is the corresponding unit corresponding eigenvector. As in (\ref{eq:EdgeSign}),
\[
\sign\left(v_\alpha\left(1\right)v_\alpha\left(2\right)\right)
=s\left(\lambda_\alpha\right)
=\sign\left(-\frac{w_{1}^{\top}G\left(\lambda_\a\right)w_{1}-d_{11}+\lambda_\a}{w_{1}^{\top}G\left(\lambda_\a\right)w_{2}-d_{12}}\right).
\]

In this section, we will prove the following:
\begin{lem}
\label{lem:Edgeslambda}
Let ${A}$ be the adjacency matrix of a $G(n,p)$ graph, and  let $\lambda_1 \ge \cdots \ge \lambda_n$ be the eigenvalues of the matrix
\[
  \t {A}= \frac{1}{\sqrt{p\left(1-p\right)(n+2)}}{A}
  - \sqrt{\frac{p}{(1-p)(n+2)}}I_{n+2}
\]
Fix $2\le\a\le\varphi_n^{\rho}$. Then
\[
\EE \left(s\left(\lambda_{\alpha} \right) \cdot \ind_{\scr T}({ A})\right)=O \left(n^{-1/3+C\eps_{LR}} \right).
\]
\end{lem}
As $\scr T$ pertains to all $n \times n$ principal submatrices, the same bound holds for \\ $\EE \left(\sign\left(v_\alpha\left(i\right)v_\alpha\left(j \right)\right) \cdot \ind_{\scr T}(\t {\ A})\right)$ for any $i\neq j$.

Once this lemma is proved, Theorem \ref{thm: edge} follows easily:
\begin{proof}
For  $2\le\a\le\varphi_n^{\rho}$, we have $\EE (\sign(u_{\alpha}(i) u_{\alpha}(j) \mid \scr{T}) = O \left(n^{-1/3+C\eps_{LR}} \right)$ for all$i \neq j$. Hence,
\[
    \EE \left((\sum_{i=1}^{n+2}\sign(u_\a(i)))^2| \scr{T}\right)= O(n^{5/3+C\eps_{LR}}).
\]
Applying Markov's inequality we get
$$
\PP \left( |\sum_{i=1}^{n+2}\sign(u_\a(i))|> n^{5/3+C'\eps} \right)<n^{-\delta_{LR}}+n^{-\eps_{LR}}.
$$
\end{proof}
The proof of this lemma will be based on the concentration we get
from Lemma \ref{lem:wGw-1}.
Let  $B$ be the $n \times n$ principal submatrix containing the last $n$ rows and columns.
If $\t {A} \in \scr T$, then $B \in \scr T_{(n,1)}$. 

Consider $\alpha = 2$ first. Let $\mu_1' \ge \mu_{n+1}'$ be the eigenvalues of the $(n+1) \times (n+1)$ matrix containing the last $(n+1)$ rows and columns of $\t {A}$.
Per \eqref{eq:B_rigidity} for $\t {A}$,
$\lambda_2 \in \left[2-n^{-2/3}\varphi_n^{2\rho},\,2+n^{-2/3}\varphi_n^{2\rho}\right]$, so interlacing and Lemma \ref{lem:LevelRepulsionForMinor}  imply that
\[
\mu_2' \le \lambda_2 \le \mu_2' + \frac{\varphi_n^C}{n} < \mu_1'
\]
where $\mu_1'$ satisfies \eqref{eq:B_LeadingEvalue}.
Repeating this argument for $B$, in view of \eqref{eq:B_LevelRepulsion} and \eqref{eq:B_LeadingEvalue}, we conclude that
$\lambda_2 \in [\mu_2, \mu_1]$.
For $2< \alpha \le \varphi_n^{\rho}$,  \eqref{eq:B_LevelRepulsion} similarly yields $\lambda_\alpha \in [\mu_\alpha, \mu_{\alpha-1}]$.

Condition on the  submatrix  $B$.
Since $\a\le\varphi_n^{\rho}$,
by the estimate that $\int_{2-t}^{2}\frac{1}{2\pi}\sqrt{4-x^{2}}\d x\ge\frac{1}{2\pi}t^{3/2}$,
we have $2-\gamma_{\a}\le n^{-2/3}\varphi_n^{\rho}$ and thus $2-\mu_{\a}\le n^{-2/3}\varphi_n^{2\rho}$
due to rigidity of eigenvalues (\ref{eq:B_rigidity}). 

Let $\scr A_{wGw}$ be the set of $n \times 2$ matrices $W$ such that
\eqref{eq: wGw estimate-1-1}   in Lemma \ref{lem:wGw-1} holds. Specifically,
 $\scr A_{wGw}$ is defined by the condition
\begin{equation}
\forall i,\,j\in\left\{ 1,\,2\right\} \;w_{i}^{\top}G\left(E\right)w_{j}=-\left(1+O\left(n^{-2\eps_{LR}}\right)\right)\delta_{ij}+\frac{\iprod{w_{i}}{u_{\a_{E}}}\iprod{w_{j}}{u_{\a_{E}}}}{\mu_{\a_{E}}-E}+O\left(n^{-1/3+C_{1}\eps_{LR}}\right)
\label{eq: wGw estimate in Lem s(lambda)}
\end{equation}
for all $E\in\left[2-n^{-2/3}\varphi_n^{2\rho},\,2+n^{-2/3}\varphi_n^{2\rho}\right]$
and a universal constant $C_{1}>0$.
Here, $\a_{E}\in\left[n\right]$ is the integer so that $\left|\mu_{\a_{E}}-E\right|\le\min_{\a\in\left[n\right]}\left|\mu_{\a}-E\right|$.

Before we move on to the proof directly, let us introduce another set.
 Let $\scr A_{W}$ be  a set of $W$
such that for $i\in\left\{ 1,\,2\right\} $
\begin{equation}
n^{-1/3+\kappa\eps_{LR}}\le\sqrt{n}\left|\iprod{\t w_{i}}{u_{\a}}\right|\le\log^2 n\label{eq:scrAdot-1-1}
\end{equation}
where $\kappa\ge\max\left\{ 2C_{1},\,8\right\} $ and
\[
\t w_{i}=w_{i} - \sqrt{\frac{p}{1-p}}l.
\]
\begin{lem}
Let the $W$ be the $n \times 2$ block $W$ of $\t {A}$ defined in \eqref{eq:ABlockDecomposition}.
With the notation above, we have
\[
\PP\left(W \in \scr A_{W}\right) \ge 1-n^{-1/3+2\kappa\eps_{LR}},
\]
and
\begin{equation}
\PP\left(\iprod{\t w_{i}}{u_{\a}}>0\right)=\frac{1}{2}+O\left(n^{-1/3+5\eps_{LR}}\right) \quad \text{for } i=1,2.\label{eq:sign}
\end{equation}

\end{lem}

\begin{proof}
The upper bound in \eqref{eq:scrAdot-1-1} holds with the desired probability due to Hoeffding's inequality.
We will estimate the probability that the lower bound holds and prove \eqref{eq:sign} at the same time.
 Let $X_{k}:=\sqrt{n+2}\t w_{1}\left(k\right)u_{\a}\left(k\right)$.
Since $\t w_{1}\left(k\right)$ has mean $0$ and variance $\frac{1}{n+2}$,
we set
\begin{align*}
S_{n} & =\frac{\sum_{k\in\left[n\right]}X_{k}}{\sum_{k\in\left[n\right]}\EE X_{k}^{2}}
=\sqrt{n+2}\iprod{\t w_{1}}{u_{\a}}
\end{align*}
Observe that $\EE X_{k}^{2}=u_{\a}\left(k\right)^{2}$ and $\EE X_{k}^{3}\le c\left(p\right)\left|u_{\a}\left(k\right)\right|^{3}$
where $c\left(p\right)>0$ is a constant depends on $p$. Let $F_{n}$
and $\Phi$ be the cumulative distributions of $S_{n}$ and the standard normal
random variable respectively. By the Berry-Esseen Theorem (see, e.g., \cite[Theorem 2.2.17]{Stroock2011})
we have
\[
\sup_{x\in\R}\left|F_{n}\left(x\right)-\Phi\left(x\right)\right|
\le C\left(\sum_{i=1}^{n}\EE X_{i}^{2}\right)^{-1/2} \cdot \max_{i}\frac{\EE\left|X_{i}\right|^{3}}{\EE X_{i}^{2}}
\le c\left(p\right)\frac{\norm{u_{\a}}_{\infty}}{\norm{u_{\a}}_{2}}.
\]

Recall that from (\ref{eq:B_l_infinity-1}) in the defintion of $\scr T_{(n,1)}$,
we have the $l_{\infty}$-norm bound: $\norm{u_{\a}}_{\infty}\le n^{-1/3+4\eps_{LR}}$.
Together with $\norm{u_{\a}}_{2}=1$ it yields
\[
\sup_{x\in\R}\left|F_{n}\left(x\right)-\Phi\left(x\right)\right|\le n^{-1/3+5\eps_{LR}}
\]
if $n$ is large enough. Thus,
\[
\PP\left(\sqrt{n}\left|\iprod{\t w_{1}}{u_{\a}}\right|\le n^{-1/3+\kappa\eps_{LR}}\right)\le\PP\left(\sqrt{n}\left|g\right|\le n^{-1/3+\kappa\eps_{LR}}\right)+2n^{-1/3+5\eps_{LR}}\le n^{-1/3+1.5\kappa\eps_{LR}},
\]
where $g\sim N\left(0,\,1\right)$ is a normal random variable. Furthermore,
we also obtain (\ref{eq:sign}) by comparing $\Phi$ and $F_{n}$.
\end{proof}
\begin{proof}[Proof of Lemma \ref{lem:Edgeslambda}]
By (\ref{eq:EdgeEvalueDetection}), if $\lambda\in\R$ is an eigenvalue
of $\t {A}$, then $\det\left(W^{\top}G\left(\lambda\right)W-D+\lambda I_{2}\right)=0$.
Let
\[
f\left(E\right):=\frac{\left(w_{1}^{\top}G\left(E\right)w_{1}-d_{11}+E\right)\left(w_{2}^{\top}G\left(E\right)w_{2}-d_{22}+E\right)}{\left(w_{1}^{\top}G\left(E\right)w_{2}-d_{12}\right)^{2}}.
\]
Thus, $\lambda$ is an eigenvalue whenever $f\left(\lambda\right)=1$. We
will use the function $f\left(E\right)$ to determine the location of the eigenvalues.

Let $\scr A_D$ be the set of all $2 \times 2$ symmetric matrices $D$ such that $\max_{i,j\in \{1,2\}}\left|d_{ij}\right|=O\left(c\left(p\right)n^{-1/2}\right)$.
Recall the definitions of $A_{wGw}$ and $A_W$
from \eqref{eq: wGw estimate in Lem s(lambda)}
and \eqref{eq:scrAdot-1-1}, respectively.
Assume that $W \in A_{wGw} \cap A_{W}$ and $D \in \scr A_D$.
We will see below that this is a likely event.

Under these conditions, the argument becomes deterministic.
By (\ref{eq:B_l_infinity-1}) from the definition of $\scr T_{(n,1)}$, 
we have $\left|\iprod{u_{\a}}l\right|\le n^{-1+2\eps_{LR}}$. Hence,
\[
\iprod{w_{i}}{u_{\a}}=\left(1+o\left(1\right)\right)\iprod{\t w_{i}}{u_{\a}}
\]
and in particular $\iprod{w_{i}}{u_{\a}}$ and $\iprod{\t w_{i}}{u_{\a}}$
have the same sign.

Observe that $E\mapsto w_{1}^{\top}G\left(E\right)w_{1}-d_{11}+E$
is a strictly increasing function on $\left(\mu_{\alpha},\,\mu_{\a-1}\right)$.
It tends to $-\infty$ as $E\rightarrow\mu_{\a}^{+}$ and $+\infty$
as $E\rightarrow \mu_{\alpha-1}^{-}$. Thus, it crosses $0$ only once. Let $E_{0}$ be
maximum of the roots of $w_{1}^{\top}G\left(E\right)w_{1}-d_{11}+E$ and
$w_{2}^{\top}G\left(E\right)w_{2}-d_{22}+E$ on $\left(\mu_{\a},\,\mu_{\a-1}\right)$.
Then by (\ref{eq: wGw estimate in Lem s(lambda)}) and $\left|d_{ij}\right|=O\left(c\left(p\right)n^{-1/2}\right)$,
\[
-\left(1+O\left(n^{-2\eps_{LR}}\right)\right)+\frac{\iprod{w_{i}}{u_{\a_{E_0}}}^2}{\mu_{\a_{E_0}}-E_0} +E_0 =0
\]
for some $i \in \{1,2\}$.
As $\mu_{\a-1}>E_0 > \mu_\a  \ge 2-n^{-2/3}\varphi_n^{2\rho}$, this implies that $E_0>\mu_{\a_{E_0}}$, and thus $\a_{E_0}=\a$.
Moreover,
 $E_0-1 = 1+O\left(n^{-2\eps_{LR}}\right)$, and so
\[
E_{0}=\left(1+O\left(n^{-2\eps_{LR}}\right)\right)\max\left\{ \iprod{w_{1}}{u_{\a}}^{2},\,\iprod{w_{2}}{u_{\a}}^{2}\right\} +\mu_{\a}.
\]

For $E>E_{0}$, both $w_{1}^{\top}G\left(E\right)w_{1}-d_{11}+E$
and $w_{2}^{\top}G\left(E\right)w_{2}-d_{22}+E$ are positive.
Setting
\[
E_{1}=2\max\left\{ \iprod{w_{1}}{u_{\a}}^{2},\,\iprod{w_{2}}{u_{\a}}^{2}\right\} +\mu_{\a},
\]
for $E\in\left[\mu_{\a},\,E_{1}\right]$, we also have $\a_{E}=\a$, and
\begin{align} \label{eq:obstacle}
\left|\frac{\iprod{w_{1}}{u_{\a}}\iprod{w_{2}}{u_{\a}}}{\mu_{\a}-E}\right|
&\ge\left|\frac{\iprod{w_{1}}{u_{\a}}\iprod{w_{2}}{u_{\a}}}{\mu_{\a}-E_{1}}\right|
=\frac{1}{2}\min\left\{ \left|\frac{\iprod{w_{1}}{u_{\a_{E}}}}{\iprod{w_{2}}{u_{\a}}}\right|,\,\left|\frac{\iprod{w_{2}}{u_{\a}}}{\iprod{w_{1}}{u_{\a}}}\right|\right\} \\
& >\log^{-2} n \cdot n^{-1/3+\kappa\eps_{LR}}. \notag
\end{align}
by (\ref{eq:scrAdot-1-1}). Hence, $w_{1}^{\top}G\left(E\right)w_{2}-d_{12}$
has no zeros in the interval $\left[\lambda_{\a},\,E_{1}\right]$. Furthermore,
because
\[
\min\left\{ \left|\frac{\iprod{w_{1}}{u_{\a}}}{\iprod{w_{2}}{u_{\a}}}\right|,\,\left|\frac{\iprod{w_{2}}{u_{\a}}}{\iprod{w_{1}}{u_{\a}}}\right|\right\} \le1,
\]
 using
 (\ref{eq: wGw estimate in Lem s(lambda)}) and $\left|d_{ij}\right|=O\left(c\left(p\right)n^{-1/2}\right)$ again, we get
\begin{align*}
\left(w_{1}^{\top}G\left(E_{1}\right)w_{2}-d_{12}\right)^{2} & =\left(\text{\ensuremath{\frac{\iprod{w_{1}}{u_{\a}}\iprod{w_{2}}{u_{\a}}}{\mu_{\a}-E_{1}}}+O\ensuremath{\left(n^{-1/3+C_{1}\eps}\right)}}\right)^{2}\\
 & =\left(\frac{1}{2}\min\left\{ \left|\frac{\iprod{w_{1}}{u_{\a}}}{\iprod{w_{2}}{u_{\a}}}\right|,\,\left|\frac{\iprod{w_{2}}{u_{\a}}}{\iprod{w_{1}}{u_{\a}}}\right|\right\} +O\left(n^{-1/3+C_{1}\eps_{LR}}\right)\right)^{2}\\
 & \le\frac{1}{4}+o\left(1\right)\le\frac{1}{2}.
\end{align*}
Together with
\begin{align*}
\left(w_{1}^{\top}G\left(E_{1}\right)w_{1}-d_{11}+E_{1}\right)\left(w_{2}^{\top}G\left(E_{1}\right)w_{2}-d_{22}+E_{1}\right) & =1+o\left(1\right)
\end{align*}
this yields
$
f\left(E_{1}\right) >1.
$
Since
$
f\left(E_{0}\right)=0,
$
 there exists $\lambda\in\left(E_{0},\,E_{1}\right)$ such
that $f\left(\lambda\right)=1$, which shows that $\lambda_\a\in\left(E_{0},\,E_{1}\right)$.

Now we will focus on $s\left(\lambda_\a\right)$.
Since $\lambda_\a>E_{0}$ , the $w_{1}^{\top}G\left(\lambda_\a\right)w_{1}-d_{11}+\lambda_\a$
is positive. Also,
\[
w_{1}^{\top}G\left(\lambda_\a\right)w_{2}-d_{12}=\frac{\iprod{w_{1}}{u_{\alpha}}\iprod{w_{2}}{u_{\alpha}}}{\mu_{\a}-\lambda_\a}+O\left(n^{-1/3+C\eps_{LR}}\right),
\]
and the magnitude of the leading term is significantly greater than $O\left(n^{-1/3+C\eps_{LR}}\right)$
by \eqref{eq:obstacle}.  Since $\mu_{\a}-\lambda_\a<0$, the expression above has the same sign
as $-\iprod{w_{1}}{u_{\a}}\iprod{w_{2}}{u_{\a}}$.
Therefore, we conclude  that
\[
s\left(\lambda_\alpha\right)
=\sign \left(-\frac{w_{1}^{\top}G\left(\lambda_\a\right)w_{1}-d_{11}+\lambda}{w_{1}^{\top}G\left(\lambda_\a\right)w_{2}-d_{12}}\right)
=\sign \left( \iprod{w_{1}}{u_{\a}}\iprod{w_{2}}{u_{\a}} \right)
=\sign \left( \iprod{\t w_{1}}{u_{\a}}\iprod{\t w_{2}}{u_{\a}} \right)
\]
for any $\t {A} \in \scr T, \ W \in \scr  A_{wGw} \cap \scr A_{W}$, and $D \in \scr A_D$.

It remains to estimate the expectation of  $s\left(\lambda_\alpha\right)$.
Recall that we conditioned on the block $B$, and $W$ and $D$ are independent of $B$. Denote this conditional expectation and probability  by $\EE_{W,\,D}$ and $\PP_{W,\,D}$.
We have
\begin{align*}
\left| \EE_{W,\,D} \left(s\left(\lambda_\a\right) \ind_{\scr T}({A}) \right) \right|
\le  & \left | \EE_{W,\,D} \left(s\left(\lambda_\a\right) \ind_{\scr T}( {A}) \ind_{\scr A_{W}}(W) \ind_{\scr A_{wGw}}(W) \ind_{\scr A_D}(D) \right)\right|\\
 & +\PP_{W,\,D}\left( W \notin \scr A_{wGw} \cup \scr A_W \right) + \PP_{W,\,D}\left( D \notin \scr A_D  \right) \\
= & \left | \EE_{W,\,D} \left(\sign \left( \iprod{\t w_{1}}{u_{\a}}\iprod{\t w_{2}}{u_{\a}} \right) \ind_{\scr T}({A}) \ind_{\scr A_{W}}(W) \ind_{\scr A_{wGw}}(W) \ind_{\scr A_D}(D) \right)\right| \\
& +O\left(n^{-1/3+C'\eps_{LR}}\right).
\end{align*}
We can get rid of the indicators in the leading term in a similar way:
\begin{align*}
 & \left | \EE_{W,\,D} \left(\sign \left( \iprod{\t w_{1}}{u_{\a}}\iprod{\t w_{2}}{u_{\a}} \right) \ind_{\scr T}({A}) \ind_{\scr A_{W}}(W) \ind_{\scr A_{wGw}}(W) \ind_{\scr A_D}(D) \right)\right| \\
\le  & \left | \EE_{W,\,D} \left(\sign \left( \iprod{\t w_{1}}{u_{\a}}\iprod{\t w_{2}}{u_{\a}} \right) \ind_{\scr T}({A}) \right)\right|
  +\PP_{W,\,D}\left( W \notin \scr A_{wGw} \cup \scr A_W \right) + \PP_{W,\,D}\left( D \notin \scr A_D  \right) \\
\le  & \left | \EE_{W,\,D} \left(\sign \left( \iprod{\t w_{1}}{u_{\a}}\iprod{\t w_{2}}{u_{\a}} \right) \ind_{\scr T}( {A}) \right)\right|
+O\left(n^{-1/3+C'\eps_{LR}}\right).
\end{align*}
Removing the conditioning over $B$, we get
\begin{align*}
& \quad \left| \EE \left(s\left(\lambda_\a\right) \ind_{\scr T}(\t {A}) \right) \right| \\
 &\le \left | \EE \left(\sign \left( \iprod{\t w_{1}}{u_{\a}}\iprod{\t w_{2}}{u_{\a}} \right) \ind_{\scr T}({A}) \right)\right|
+O\left(n^{-1/3+C'\eps_{LR}}\right) \\
&\le \left | \EE \left(\sign \left( \iprod{\t w_{1}}{u_{\a}}\iprod{\t w_{2}}{u_{\a}} \right) \ind_{\scr T_{(n,1)}}({A}^{(1,2)}) \right)\right|
  +\EE \left|  \ind_{\scr T}({A}) - \ind_{\scr T_{(n,1)}}(\t { A}^{(1,2)} ) \right|  +O\left(n^{-1/3+C'\eps_{LR}}\right)
\end{align*}
In view of Corollary \ref{cor:typical submatrices}, the second term does not exceed $n^{-1/3+2\eps_{LR}}$.
To bound the first term, we condition again on the block $B=\t { A}^{(1,2)}$ such that $\t {A}^{(1,2)} \in \scr T_{(n,1)}$ and apply \eqref{eq:sign}.
By this inequality,
 \[
 P_{i}:=\PP\left[\iprod{\t w_{1}}{u_{a}}\ge0 \mid \t { A}^{(1,2)}\right]:=\frac{1}{2}+p_{i}
 \]
where $p_{i}=O\left(n^{-1/3+5\eps_{LR}}\right)$.
Using the independence of $\t w_{1}$ and $\t w_{2}$, we get
\begin{align*}
\EE\left[\sign\left(\iprod{\t w_{1}}{u_{a}}\iprod{\t w_{2}}{u_{\a}}\right) \mid \t {A}^{(1,2)} \right]
 & =P_{1}P_{2}+(1- P_{1})  (1-P_{2})-P_{1}(1-P_{2})-(1- P_{1}) P_{2}\\
 & =4p_{1}p_{2}=O\left(n^{-2/3+10\eps_{LR}}\right).
\end{align*}
Removing the conditioning completes the proof of Lemma \ref{lem:Edgeslambda}.
\end{proof}

\section{Appendix}

In this section we establish the spectral properties of symmetric random matrices appearing in Definition \ref{def: Ank}.
Namely, we prove the following lemma:
\begin{lem}
  Fix $p\in (0,1)$, $D>0$. Let $H_p$ be a symmetric $n\times n$ matrix with zero diagonal and i.i.d entries above the diagonal.
  The non-diagonal entries have the distribution:
\[
  h_{ij}=\begin{cases}
\sqrt{\frac{1-p}{p}}\frac{1}{\sqrt{n}} & \text{with probability }p,\\
-\sqrt{\frac{p}{1-p}}\frac{1}{\sqrt{n}} & \text{with probability }1-p.
\end{cases}
\]
  Then, $H_p$ satisfies \eqref{eq:IsotropicGreenFunction} -- \eqref{eq:M_evectorl_infinity} with probability greater than $1-n^{-D}$.
Furthermore, for a sufficiently small $\eps>0$, there exists $\delta>0$ such that $H_p\in \scr{LR}\left( n, n^{-2/3-\eps}\right)$ with probability greater than
$1-n^{-\delta}$.
\end{lem}
Note that condition \eqref{eq:LevelRepulsion} involving $\theta$ and $k$ can be derived from the second part of this lemma by appropriately adjusting $\eps$.

Conditions  \eqref{eq:IsotropicGreenFunction} and \eqref{eq:M_evectorl_infinity}
were derived in \cite[Theorem 2.1, 2.2]{Erdos2012b} and
Conditions \eqref{eq:M_rigidity} and  \eqref{eq:M_evectorl_infinity} were proved in \cite[Theorem 2.12, 2.16]{EJP3054}.

Condition  \eqref{eq:LevelRepulsion} was proved in  \cite{Knowles2013a}. However, the matrix model is slighly different from ours.
To show that $H_p$ satisfies level repulsion at the edge, we rely
on the fact that GOE satisfies this condition and apply
Green Function Comparison Theorem. This strategy is stated as Proposition
2.4 in \cite{Knowles2013a}:

\begin{prop} \label{KYlevelRepulsion}
\label{prop:LR}Let $H^{v}$ and $H^{w}$ be $n \times n$
symmetric random matrices with independent
entries $h_{ij}^v$ and $h_{ij}^w$ such that the $\EE h_{ij}^v = \EE h_{ij}^w = 0$ and
$\EE (h_{ij}^v)^2=\EE (h_{ij}^w)^2 = \sigma_{ij}^2$.
Assume that $\Sigma = (\sigma_{ij})$ satisfies the following conditions
\begin{enumerate}
\item For $j\in [n]$, $\sum_{i=1}^n \sigma_{ij}^2 = 1$.
\item There exists $\delta_W>0$ such that $1$ is a simple eigenvalue of $\Sigma$ and $\rm{Spec}(\Sigma) \sub [-1+\delta_W, 1-\delta_W] \cup \{1\}$.
\item There is a constant $C_W$, independent of $n$, such that $\max_{ij}\{\sigma_{ij}^2\}\le \frac{C_W}{n}$.
\end{enumerate}
Also, assume that $h_{ij}$ have a uniformly subexponential decay. Namely, there exists a constant $\nu>0$, independent of $n$, such that for any $x\ge1$ and
$1 \le i,j \le n$ we have
\[
	\PP\left( |h_{ij}| > x\sigma_{ij} \right) \le \nu^{-1}\exp(-x^\nu).
\]
Assume that $H^{v}$ satisfies the \emph{Level Repulsion Condition}, i.e. for a sufficiently small $\eps>0$,
 there exists $\delta>0$ such that $H^{v} \in \scr{LR}\left( n, n^{-2/3-\eps}\right)$ with probability greater than
$1-n^{-\delta}$. Then the same holds for $H^w$ with a different $\delta=\delta(\eps)$.

\end{prop}

The level repulsion condition has been proved for the GOE ensemble, see, e.g. \cite{anderson_guionnet_zeitouni_2009}.
By GOE we mean that a $n\times n$ symmetric random matrix $W$ with independent centered gaussian entries (up to symmetry) where
the off-diagonal entries have variance $1/n$ and the diagonal entries have variance $2/n$.
We would like to apply Proposition \ref{KYlevelRepulsion} with  $H^{v}=W$ and $H^{w}=
H_{p}$. The first two moments of the off-diagonal entries of these two ensembles are the same.
The variances of the diagonal entries differ, but since there are only $n$ of them, it will be possible to show that they do not affect the level repulsion significantly.

We proceed in two steps. First, we prove the level repulsion condition for a $n\times n$ matrix $\t W$ whose off diagonal entries are the same as for $W$ and the diagonal entries are $0$.
Then, we apply Proposition  \ref{KYlevelRepulsion} to $H^v=\t W$ and $H^w= H_p$.

Thus, it is sufficient to prove
\begin{prop}
\label{prop:LR GOEtoWigner}The level repulsion estimates hold for $\t W$.
\end{prop}

The proof of this proposition is standard and is included it for the reader's convenience.
It follows the proof of  \ref{prop:LR} which relies
on Lemma 2.6 (Green Function Comparison Theorem) and Lemma 2.7 in \cite{Knowles2013a}.

Since the second moments of the diagonal entries of $W$ and $\t W$ differ, we need a substitute for
Green Function Comparison Theorem. The rest of the proof will be exactly the same
as of Proposition \ref{prop:LR}.

Before stating the result precisely, we will sketch the idea behind the comparison.
Consider the Stieltjes Transform of a symmetric matrix $H$ is $m\left(z\right)=\frac{1}{n}\Tr\left(\frac{1}{H-z}\right)$.
Suppose $\lambda_{1},\dots,\,\lambda_{n}$ are eigenvalues of $H$.
Then,
\[
\frac{n}{\pi}\Im{m\left(E+\i\eta\right)}=\sum_{i\in\left[n\right]}\frac{1}{\pi}\frac{\eta}{\left(\lambda_{i}-E\right)^{2}+\eta^{2}}.
\]
If we choose $\eta$ to be sufficiently small, then each summand is
an approximation of the delta function at each eigenvalue.
On one hand, this provides a way to estimate number of eigenvalues
in an interval. Taking $\eta$ to be sufficiently small, we should
have
\[
 \sum_{\a}^n \ind_{(a,b)}(\lambda_\a) \simeq n \int_{a}^{b}\frac{1}{\pi}\Im{m\left(E+\i\eta\right)}\d E.
\]
On the other hand, $\Im{m\left(E+\i\eta\right)}$ can be expressed
in terms of the Green Function $G\left(z\right):=\frac{1}{H-z}.$
\[
\Im{m\left(E+\i\eta\right)=\frac{1}{n}\sum_{i}\Im{G_{ii}\left(E+\i\eta\right)}}.
\]
We will use Lindeberg's method to replace the diagonal entries of $W$ by those of $\t W$ one by one
and estimate the expectation of the difference of Green functions.

Now we state the substitute for Lemma 2.6 in \cite{Knowles2013a}:
\begin{lem}
\label{lem:GFCT}(Green Function Comparison Theorem) Let $F:\R\mapsto\R$
be a bounded smooth function whose first and second derivatives are  bounded as well.
There exists a constant $\eps_{0}>0$ and for
such $\eps<\eps_{0}$ and for any real numbers $E_{1},\,E_{2}\in\left[2-n^{-2/3+\eps},2+n^{2/3+\eps}\right]$,
setting $\eta=n^{-2/3-\eps}$ we have
\[
\left|\left(\EE^{W}-\EE^{\t W}\right)F\left(n\int_{E_{1}}^{E_{2}}\Im{m\left(y+i\eta\right)}\d y\right)\right|\le cn^{-1/3+c\eps}.
\]
\end{lem}

Lindeberg's method is based on replacing the entries one by one.
Yet, our proof uses the strong local semicircle law, see Theorem \ref{thm: StrongSemiCircle} below.
Application of this law requires scaling of the matrix so that the variance matrix will be doubly stochastic.
However, replacing diagonal entries of $W$ by $0$ appearing in $\t W$ results in two  essentially different
scalings of the variance matrix to the doubly stochastic form. To deal with this obstacle, we perform replacement
in smaller steps which will require $n^2$ steps instead of $n$.

Define $n^2$ symmetric random matrices
$\left\{ W_{\b,\,\gamma}\right\} _{\b,\,\gamma=0}^{n}$ whose off-diagonal entries
are the same as of $W$ and $\t W$. Let  $\left\{ h_{i,j}\right\} _{i,j=1}^{n}$
be i.i.d $N\left(0,\frac{2}{n^{2}}\right)$ random variables. The diagonal entries of $W_{\b,0}$ are
\[
  \left(W_{\b,0}\right)_{i,i}=\sum_{j=1}^{\b}h_{j,i}.
\]
In particular, the diagonal entries of $W_{\b,0}$ are centered gaussian
  variables with variance $\frac{2\b}{n^2}$. Thus, the variance matrix of
  $W_{\b,0}$ is doubly stochastic
if we scale it by a factor $1+O(n^{-1})$.
Furthermore, $W_{0,0}=\t W$ and $W_{n,0}=W$.

Now we define the diagonal entries of $W_{\b,\gamma}$:
\[
  \left(W_{\b,\gamma}\right)_{ii} =
    \begin{cases}
        \sum_{j=1}^\b h_{j,i} & \text{if $i>\gamma$},\\
        \sum_{j=1}^\b h_{j,i} + h_{\b+1,i} & \text{if $i\le\gamma$}.
    \end{cases}
\]
In other words, we have
\[
W_{\b,\gamma+1}=W_{\b,\gamma}+h_{\b+1,\gamma+1}e_{\gamma+1}e_{\gamma+1}^{\top}
\]
and
\[
  W_{\b,n}=W_{\b+1,0}.
\]

Our goal is to show that
\[
\left|
  \left(\EE^{W_{\b,\gamma}}-\EE^{W_{\b,\gamma+1}}\right)
  F\left(
     n\int_{E_{1}}^{E_{2}}\Im{m\left(y+i\eta\right)}\d y
  \right)
\right|
\le n^{-2}n^{-1/3+c\eps}
\]
for each $k =0,\dots, n-1$ and $\gamma=0,\dots, n-1$.
Then the statement of the theorem will follow immediately.
Before we move on to the proof, we need the following proposition.
\begin{prop}
\label{prop:SSLforall} Fix a sufficiently small $\eps>0$.
Let $\scr I:=\left\{ E+\i\eta\,:\,\left|E-2\right|<n^{-2/3+\eps}\right\} $
and $\eta=n^{-2/3-2\eps}$.Then, for any $D>0$, if $n$ is sufficiently
large, we have
\[
\PP\left(\max_{\b,\gamma}\sup_{z\in\scr I}\left|\left(G_{\b,\gamma}\left(z\right)\right)_{ij}-\delta_{ij}\right|>n^{-1/3+4\eps}\right)<n^{-D}
\]

where $G_{\b,\gamma}\left(z\right)=\frac{1}{W_{\b,\gamma}-z}$ is the
Green function of $W_{\b,\,\gamma}$.
\end{prop}

Let's recall a theorem in \cite[Theorem 2.1]{Erdos2012b}.

\begin{thm}
\label{thm: StrongSemiCircle}(Strong local semicircular law) Suppose that
$H$ satisfies the assumption of Proposition \ref{KYlevelRepulsion}.
Then, for every $s,\,D>0$ and $0<\eps<1/3$, we have
\begin{equation}
\PP\left(\sup_{\left|E-2\right|\le n^{-2/3+\eps}}\max_{i,j\in\left[n\right]}\left|\left(G\left(E+\i\eta\right)\right)_{ij}-1\right|<4n^{-\frac{1}{3}+s+\eps}\right)\ge1-n^{-D}\label{eq:isotropicdelocalization}
\end{equation}
where $\eta=n^{-2/3-\eps}$ and $n\ge n\left(s,\,D, \eps \right)$.
\end{thm}

This theorem implies that $\max_{\b}\sup_{z\in\scr I}\left|\left(G_{\b,0}\left(z\right)\right)_{ij}-\delta_{ij}\right|\le4n^{-\frac{1}{3}+3\eps}$
with probability at least $1-n^{-D}$. We extend this
properties to $W_{\b,\gamma}$ by comparison.
\begin{proof}[Proof of Proposition \ref{prop:SSLforall}]
Fix $\b$. Fix a sample of $W_{\b,\,0}$ such that
\[
\sup_{\left|E-2\right|\le n^{-2/3+\eps}}\max_{i,j\in\left[n\right]}\left|\left(G\left(E+\i\eta\right)\right)_{ij}-1\right|<4n^{-\frac{1}{3}+3\eps}
\]
 for $\left|E-2\right|\le n^{-2/3+\eps}$ and the samples of
$\left\{ h_{\b+1,\gamma}\right\} _{\gamma=1}^{n}$ such that $\max_{\gamma}\left|h_{\b+1,\,\gamma}\right|\le\frac{\varphi_{n}}{n}$
where $\varphi_{n}=\left(\log n\right)^{\log\log n}$. Notice that
both conditions hold with probability at least $1-n^{-D}$.

Define $s_{0}=4n^{-\frac{1}{3}+3\eps}$ and
$s_{\gamma+1}=s_{\gamma}\left(1+\frac{1}{\varphi_{n}n}\right)$.
We claim that
\begin{equation}
\left|\left(G_{\b,\gamma}\left(E+\i\eta\right)\right)_{i,j}-\delta_{ij}\right|\le\text{\ensuremath{\phi\left(i,j,\gamma\right)}}s_{\gamma}\label{eq:Ggamma}
\end{equation}
where
\[
\phi\left(i,j,\gamma\right):=1+\ind_{i\ge\gamma}+\ind_{j\ge\gamma}.
\]
If it is true, then we have
\[
  \max_{\b,\gamma}\sup_{z\in\scr I}
    \left|
      \left(G_{\b,\gamma}\left(E+\i\eta\right)\right)_{ij}-\delta_{ij}
    \right|
  \le
  3s_{n}\le 3s_0(1+\frac{1}{\varphi_n n})^n
  \le n^{-1/3+4\eps}.
\]

If the matrices $A$ and
$A+B$ are invertible, then the following resolvent identity holds:
\[
\frac{1}{A+B}=\frac{1}{A}-\frac{1}{A}B\frac{1}{A+B}.
\]

Applying the equality repeatedly we get
\[
\frac{1}{A+B}=\frac{1}{A}-\frac{1}{A}B\frac{1}{A}+\frac{1}{A}B\frac{1}{A}B\frac{1}{A}-\left(\frac{1}{A}B\right)^{3}\frac{1}{A}\cdots\pm\left(\frac{1}{A}B\right)^{k}\frac{1}{A+B}.
\]

Suppose that (\ref{eq:Ggamma}) holds up to $\gamma-1$. Let $A=W_{\b,\,\gamma-1}-\left(E+\i\eta\right)I_{n}$
and $B=h_{\b+1,\gamma}e_{\gamma}e_{\gamma}^{\top}$. For simplicity,
we write
\[
h=h_{\b+1,\gamma},\ P=e_{\gamma}e_{\gamma}^{\top},\ R=\frac{1}{A}=G_{\b,\gamma-1}\left(E+\i\eta\right),\,\text{and }
S=\frac{1}{A+B}=G_{\b,\gamma}\left(E+\i\eta\right).
\]
 The equality above can be written as
\[
S=\frac{1}{A+B}=R-hRPR+h^{2}\left(RP\right)^{2}R+\dots h^{k}\left(RP\right)^{k}S.
\]
Entry-wise, we have
\begin{align}
  S_{ij}
=&
  R_{ij}
  - h R_{i\gamma}R_{\gamma j}
  + h^2 R_{i\gamma}R_{\gamma\gamma} R_{\gamma j}
  \dots (-1)^k h^k R_{i\gamma} R^{k-1}_{\gamma\gamma} S_{\gamma j} \notag \\
=&
  R_{ij}
  - h R_{i\gamma}R_{\gamma j}
  \left(
    \sum_{l=0}^k\left(-h R_{\gamma \gamma}\right)^l
  \right)
  + (-1)^k h^k R_{i\gamma} R^{k-1}_{\gamma\gamma} S_{\gamma j} \label{eq:SR1}
\end{align}

We will  use the following uniform bound of the entries of $S$:
\[
	|S_{\gamma j}| \le \|S\| = \left\|\frac{1}{W_{\beta,\gamma}(E+\i \eta)}\right\| \le \frac{1}{\eta} \le n^{2/3+\eps}.
\]
Together with $|h|<\frac{\varphi_n}{n}$ and $\max\left\{ |R_{i\gamma}|, \, |R_{\gamma j}|\right\}\le 1+ s_\gamma \le 2$,
this means that the last summand in \eqref{eq:SR1} is less
than $\frac{1}{n^{3}}$ if we pick $k=5$. From now on we will fix $k=5$.
Then,
\[
  \left|
    \sum_{l=0}^k\left(-h R_{\gamma \gamma}\right)^l
  \right|
  \le C.
\]
for some absolute constant $C>0$. Therefore,
\begin{align*}
\left|S_{ij}-\delta_{ij}\right| & \le\left|R_{ij}-\delta_{ij}\right|+C\left|hR_{i\gamma}R_{j\gamma}\right|+\frac{1}{n^{3}}\\
 & \le\phi\left(i,j,\gamma-1\right)s_{\gamma-1}+C\left|hR_{i\gamma}R_{\gamma j}\right|+\frac{1}{n^{3}}.
\end{align*}

It remains to show
\[
\phi\left(i,j,\gamma-1\right)s_{\gamma-1}+C\left|hR_{i\gamma}R_{j\gamma}\right|+\frac{1}{n^{3}}\le\phi\left(i,j,\gamma\right)s_{\gamma}.
\]
Consider $\gamma\notin\left\{ i,j\right\} $. We use the bound $\left|R_{i\gamma}\right|\le3s_{\gamma-1}$
and $\left|R_{\gamma j}\right|\le3s_{\gamma-1}<\frac{1}{\varphi_{n}^{3}}$
to get
\begin{align*}
C\left|hR_{i\gamma}R_{j\gamma}\right|+\frac{1}{n^{3}} & \le s_{\gamma-1}\frac{C}{n\varphi_{n}^{2}}+\frac{1}{n^{3}}\le s_{\gamma-1}\frac{1}{n\varphi_{n}}.
\end{align*}
Therefore, we have
\begin{align*}
  \left|S_{ij}-\delta_{ij}\right|
\le&
  \phi\left(i,j,\gamma-1\right)s_{\gamma-1}+\frac{1}{n\varphi_{n}}s_{\gamma-1}\\
\le&
  \phi\left(i,j,\gamma-1\right)s_{\gamma-1}\left(1+\frac{1}{\varphi_n n}\right)\\
\le&
  \phi\left(i,j,\gamma\right)s_{\gamma}
\end{align*}

In the case $\gamma\in\left\{ i,\,j\right\} $, we use the trivail
bounds that $\max\left\{ \left|R_{i\gamma}\right|,\left|R_{j\gamma}\right|\right\} \le1+3s_{\gamma-1}\le2$.
Thus, we have
\[
C\left|hR_{i\gamma}R_{j\gamma}\right|+\frac{1}{n^{3}}\le\frac{4\varphi_{n}}{n}+\frac{1}{n^{3}}\le s_{0}.
\]
Notice that $\phi\left(i,j,\gamma\right)-\phi\left(i,j,\gamma-1\right)\ge1$ since $\gamma \in \{ i,\, j \}$.
\[
\left|S_{ij}-\delta_{ij}\right|\le\phi\left(i,j,\gamma-1\right)s_{\gamma-1}+s_{0}\le\phi\left(i,j,\gamma\right)s_{\gamma}.
\]
The result follows.
\end{proof}
Now we are ready to prove Lemma \ref{lem:GFCT}.
\begin{proof}
Recall that our goal is to show  that
\[
\left|
  \left(\EE^{W_{\b,\gamma}}-\EE^{W_{\b,\gamma+1}}\right)
  F\left(
     n\int_{E_{1}}^{E_{2}}\Im{m\left(y+i\eta\right)}\d y
  \right)
\right|
\le n^{-2}n^{-1/3+c\eps}
\]
With probability greater than $1-n^{-D}$,
we have
\[
\sup_{\left|E-2\right|\le n^{-2/3+\eps}}\left|\left(G_{\b,\gamma}\left(E+\i\eta\right)\right)_{ij}-\delta_{ij}\right|\le n^{-1/3+\eps}.
\]
for $\b= 0, \dots,\, n-1$ and $\gamma = 0, \dots, \, n-1$. Now, we fix $\b$ and $\gamma$.  Fix a sample of $W_{\b,\gamma-1}$
such that the above inequality holds.

We recycle the notation from the proof of Proposition \ref{prop:SSLforall}.
Let $A=W_{\b,\,\gamma-1}-\left(E+\i\eta\right)I_{n}$ and $B=h_{\b+1,\gamma}e_{\gamma}e_{\gamma}^{\top}$.
For simplicity, we write
\[
h=h_{\b+1,\gamma},\ P=e_{\gamma}e_{\gamma}^{\top},\,\, R=\frac{1}{A}=G_{\b,\gamma-1}\left(E+\i\eta\right), \text{ and }
S=\frac{1}{A+B}=G_{\b,\gamma}\left(E+\i\eta\right).
\]
Then,
\[
S_{ij}=R_{ij}+hR_{i \gamma}R_{j \gamma}+h^{2}R_{i \gamma}R_{\gamma \gamma}R_{j \gamma}+h^{3}R_{i \gamma}R_{\gamma \gamma}^{2}S_{j \gamma},
\]
where, as before, $|S_{j \gamma}| \le \norm{S} \le n^{2/3+ \eps}$.
Taking expectation with respect to $h$ and using $|R_{i \gamma}| \le n^{-1/3+ \eps} + \delta_{i \gamma}$, we get
\[
\left|\EE_{h}S_{ii}-R_{ii}\right| \le \frac{2}{n^{2}}n^{-2/3+2\eps}+\frac{C}{n^{3}} n^{1/3+ 2 \eps}+\delta_{i \gamma}\frac{C}{n^2}.
\]
Furthermore, if  $\left|h\right|\le\frac{\varphi_{n}}{n}$, then
by (\ref{eq:SR1})
\[
\left|S_{ii}-R_{ii}\right|\le C\left|hR_{i\gamma}R_{\gamma i}\right|+\frac{1}{n^{3}}\le\varphi_{n}n^{-5/3+3\eps}+\delta_{i \gamma}\frac{\varphi_{n}}{n}.
\]
Therefore,
\begin{align}
\text{\ensuremath{\left|\sum_{i=1}^{n}\left(S_{ii}-R_{ii}\right)\right|}} & \le n^{-2/3+4\eps}
\quad \text{when } \left|h\right|\le\frac{\varphi_{n}}{n},
\label{eq:ErrorWithoutExpectation}
\intertext{and}
\left|\EE_{h}\sum_{i=1}^{n}\left(S_{ii}-R_{ii}\right)\right| & \le n^{-5/3+3\eps}.\label{eq:ErrorWithExpectation}
\end{align}
Now we examine the difference:
\begin{align*}
 & F\left(\int_{E_{1}}^{E_{2}}\sum_{i}S_{ii}\left(y+\i\eta\right)\d y\right)-F\left(\int_{E_{1}}^{E_{2}}\sum_{i}R_{ii}\left(y+\i\eta\right)\d y\right)\\
= & F'\left(\int_{E_{1}}^{E_{2}}\sum_{i}R_{ii}\left(y+\i\eta\right)\d y\right)
\left(\int_{E_{1}}^{E_{2}}\sum_{i}\left(S_{ii}(y+\i\eta)-R_{ii}(y+\i\eta)\right)\d y\right)+\\
 & O\left(\left(\int_{E_{1}}^{E_{2}}\sum_{i}\left(S_{ii}(y+\i\eta)-R_{ii}(y+\i\eta)\right)\d y\right)^{2}\right)
\end{align*}
where we rely on the fact that $F''$ is bounded. Since $\left|E_{2}-E_{1}\right|\le2n^{-2/3+\eps}$, by \eqref{eq:ErrorWithoutExpectation}
we have
\[
\left(\int_{E_{1}}^{E_{2}} \sum_{i} \left(  S_{ii}(y+\i\eta)-R_{ii}(y+\i\eta) \right) \, \d y\right)^{2}\le\left(2n^{-2/3+\eps}n^{-2/3+4\eps}\right)^{2}\le n^{-8/3+C\eps}
\]
if $\left|h\right|\le\frac{\varphi_{n}}{n}$. Furthermore, if we take
the expectation with respect to $h$ and $W_{\b,\gamma}$, the same bound still
holds. Indeed, we can apply this bound conditioning on $\left|h\right|\le\frac{\varphi_{n}}{n}$, and use a trivial bound
\[
\left(\int_{E_{1}}^{E_{2}}\sum_{i}S_{ii}(y+\i\eta)-R_{ii}(y+\i\eta) \d y\right)^{2}\le n^{C}
\]
valid with some fixed constant $C>0$ for other $h$.
Similarly, \eqref{eq:ErrorWithExpectation} yields
\[
\left|\EE_{W_{\b,\,\gamma-1}}\EE_{h}\left(\int_{E_{1}}^{E_{2}}\sum_{i}\left(S_{ii}(y+\i\eta)-R_{ii}(y+\i\eta)\right) \d y\right)\right|\le n^{-7/3+C\eps}.
\]

Therefore, we conclude that
\[
\left(\EE^{W_{\b,\gamma-1}}-\EE^{W_{\b,\gamma}}\right)\Im{m\left(E+\i\eta\right)}\le n^{-2}n^{-1/3+c\eps}
\]
finishing the proof.
\end{proof}

\bibliographystyle{plain}

\begin{bibdiv}
  \begin{biblist}
    \bib{anderson_guionnet_zeitouni_2009}{book}{
        author={Anderson, Greg~W},
        author={Guionnet, Alice},
        author={Zeitouni, Ofer},
         title={{An Introduction to Random Matrices}},
        series={Cambridge Studies in Advanced Mathematics},
     publisher={Cambridge University Press},
          date={2009},
  }

  \bib{Arora2012EigenvectorsOR}{inproceedings}{
        author={Arora, Sanjeev},
        author={Bhaskara, Aditya},
         title={{Eigenvectors of Random Graphs: Delocalization and Nodal
    Domains}},
          date={2012},
  }

  \bib{BHY19}{article}{
    author={Bauerschmidt, Roland},
    author={Huang, Jiaoyang},
    author={Yau, Horng-Tzer},
    title={{Local Kesten--McKay Law for Random Regular Graphs}},
    journal={Communications in Mathematical Physics},
    year={2019},
    volume={369},
    number={2},
  pages={523 \ndash 636},
}



  \bib{EJP3054}{article}{
        author={Bloemendal, Alex},
        author={Erd{\"o}s, L{\'{a}}szl{\'{o}}},
        author={Knowles, Antti},
        author={Yau, Horng-Tzer},
        author={Yin, Jun},
         title={{Isotropic local laws for sample covariance and generalized
    Wigner matrices}},
          date={2014},
          ISSN={1083-6489},
       journal={Electron. J. Probab.},
        volume={19},
         pages={no. 33, 1\ndash 53},
           url={http://ejp.ejpecp.org/article/view/3054},
  }

  \bib{BY17}{article}{
        author={Bourgade, Paul},
        author={Yau, Horng-Tzer},
         title={{The eigenvector moment flow and local quantum unique ergodicity}},
          date={2017},
       journal={Comm. Math. Phys.},
        volume={350},
        number={1},
         pages={231 \ndash 278},
  }

  \bib{bourgade2017}{article}{
        author={Bourgade, Paul},
        author={Huang, Jiaoyang},
        author={Yau, Horng-Tzer},
         title={{Eigenvector statistics of sparse random matrices}},
          date={2017},
       journal={Electron. J. Probab.},
        volume={22},
         pages={38 pp.},
           url={https://doi.org/10.1214/17-EJP81},
  }

  \bib{BYY19}{article}{
        author={Bourgade, Paul},
        author={Yau, Horng-Tzer},
        author={Yin, Jun},
         title={{Random band matrices in the delocalized phase: 
         Quantum unique ergodicity and universality}},
        year ={to appear},
       journal={Comm. Pure Appl. Math},

  }


  \bib{CH89}{book}{
        author={Courant, Richard},
        author={Hilbert, David},
         title={{Methods of Mathematical Physics.}},
     publisher={Wiley-VCH},
          date={1989},
        volume={1},
          ISBN={9780471504474},
    url={http://proxy.lib.umich.edu/login?url=http://search.ebscohost.com/login.aspx?direct=true{\&}db=nlebk{\&}AN=26283{\&}site=ehost-live{\&}scope=site},
  }

  \bib{DLL11}{article}{
        author={Dekel, Yael},
        author={Lee, James~R.},
        author={Linial, Nathan},
         title={{Eigenvectors of random graphs: Nodal Domains}},
          date={2011},
          ISSN={10429832},
       journal={Random Structures and Algorithms},
        volume={39},
        number={1},
         pages={39\ndash 58},
        eprint={0807.3675},
  }

  \bib{EKYY12}{article}{
        author={Erd{\"o}s, L{\'{a}}szl{\'{o}}},
        author={Knowles, Antti},
        author={Yau, Horng-Tzer},
        author={Yin, Jun},
         title={{Spectral statistics of Erd{\"o}s-R{\'{e}}nyi graphs II: 
          Eigenvalues spacing and the extreme eigenvalues}},
          date={2012},
       journal={Comm. Math. Phys.},
        volume={314},
        number={3},
         pages={587\ndash 640},
  }


  \bib{Erdos2012b}{article}{
        author={Erd{\"o}s, L{\'{a}}szl{\'{o}}},
        author={Yau, Horng-Tzer},
        author={Yin, Jun},
         title={{Rigidity of eigenvalues of generalized Wigner matrices}},
          date={2012},
          ISSN={0001-8708},
       journal={Advances in Mathematics},
        volume={229},
        number={3},
         pages={1435\ndash 1515},
    url={https://www.sciencedirect.com/science/article/pii/S0001870811004063?via{\%}3Dihub},
  }

  \bib{Erdos2013}{article}{
        author={Erd{\"o}s, L{\'{a}}szl{\'{o}}},
        author={Knowles, Antti},
        author={Yau, Horng-Tzer},
        author={Yin, Jun},
         title={{Spectral statistics of Erd{\"o}s-R{\'{e}}nyi graphs I: local
         semicircle law }},
          date={2013},
          ISSN={0001-8708},
       journal={Ann. Probab.},
        volume={41},
        number={3},
         pages={2279 \ndash 2375},
  }

  \bib{HLY15}{article}{
    author={Huang, Jiaoyang},
    author={Landon,Benjamin},
    author={Yau,Horng-Tzer},
    title={{Bulk universality of sparse random matrices}},
    journal ={Journal of Mathematical Physics},
    volume ={56},
    number = {12},
    pages = {123\ndash 301},
year = {2015},
  }

  \bib{Knowles2013a}{article}{
        author={Knowles, Antti},
        author={Yin, Jun},
         title={{
           The isotropic semicircle law and
           deformation of Wigner matrices
         }},
          date={2013},
       journal={
         Comm. Pure Appl. Math.
       },
        volume={66},
        number={11},
         pages={1663 \ndash 1770},
  }

  \bib{KY13}{article}{
        author={Knowles, Antti},
        author={Yin, Jun},
         title={{Eigenvector distribution of Wigner matrices}},
          date={2013},
          ISSN={01788051},
       journal={Probability Theory and Related Fields},
        volume={155},
        number={3-4},
         pages={543\ndash 582},
  }
  \bib{LS18}{article}{
    author={Lee, Ji Oon},
    author={Schnelli, Kevin},
    title={{Local law and Tracy-Widom limit for sparse random matrices}},
    journal={Probability Theory and Related Fields},
    year ={2018},
    volume = {171},
    number = {1},
    pages={543\ndash 616},
}


  \bib{LPRTV05}{article}{
        author={Litvak, Alexander},
        author={Pajor, A},
        author={Rudelson, M},
        author={Tomczak-Jaegermann, Nicole},
        author={Vershynin, R},
         title={{Euclidean embeddings in spaces of finite volume ratio via random
    matrices}},
          date={2005},
       journal={J. Reine Angew.
    Math.},
        volume={2005},
         pages={1\ndash 19},
  }

  \bib{Nguyen2017}{article}{
        author={Nguyen, Hoi},
        author={Tao, Terence},
        author={Vu, Van},
         title={{Random matrices: tail bounds for gaps between eigenvalues}},
          date={2017},
          ISSN={1432-2064},
       journal={Probability Theory and Related Fields},
        volume={167},
        number={3},
         pages={777\ndash 816},
           url={https://doi.org/10.1007/s00440-016-0693-5},
  }

  \bib{Rudelson2017}{article}{
        author={Rudelson, Mark},
         title={{Delocalization of eigenvectors of random matrices}},
       journal={submitted},
        eprint={arXiv:1707.08461},
           url={http://arxiv.org/abs/1707.08461},
  }

  \bib{Rudelson2013}{article}{
        author={Rudelson, Mark},
        author={Vershynin, Roman},
         title={{Hanson-Wright inequality and subgaussian concentration}},
          date={2013},
          journal={Electron. Commun. Probab.},
        volume={18},
        number={82},
        pages={9 pp.},
  }

  \bib{Rudelson2016}{article}{
        author={Rudelson, Mark},
        author={Vershynin, Roman},
         title={{No-gaps delocalization for general random matrices}},
          date={2016},
          ISSN={1016443X},
       journal={Geometric and Functional Analysis},
       volume={26},
       number={6},
         pages={1716 \ndash 1776},
  }

  \bib{Stroock2011}{book}{
        author={Stroock, Daniel W.},
         title={{Probability theory. An analytic view
          Second edition}},
     publisher={Cambridge Univ. Press},
       address={Cambridge},
          date={2011},
  }

  \bib{TVW13}{article}{
    author={Tran, Linh V.},
    author={Vu, Van H.},
    author={Wang, Ke},
    title={{Sparse random graphs: Eigenvalues and eigenvectors}},
    journal={Random Structures \& Algorithms},
    volume = {42},
    number = {1},
    pages = {110 \ndash 134},
    year = {2013},
  }

  \bib{Vershynin2018}{book}{
        author={Vershynin, Roman},
         title={{High-Dimensional Probability: An Introduction with Applications
    in Data Science}},
     publisher={Cambridge Univ. Press},
       address={Cambridge},
          date={2018},
  }

  \bib{Wig1955}{article}{
        author={Wigner, Eugene P.},
         title={{Characteristic vectors of bordered matrices with infinite dimensions}},
          date={1955},
       journal={Ann. of Math.},
       volume={62},
       number={3},
         pages={548 \ndash 564},
  }


  \bib{Zelditch2017}{book}{
        author={Zelditch, Steve},
         title={{Eigenfunctions of the Laplacian of a Riemannian manifold. CBMS
    Regional Conference Series in Mathematics, vol. 125}},
     publisher={American Mathematical Society},
       address={Providence, RI},
          date={2017},
          ISBN={9781470410377},
           url={{\%}3C?xml version={\%}221.0{\%}22
    encoding={\%}22UTF-8{\%}22?{\%}3E{\%}3Ccollection{\%}3E{\%}3Crecord{\%}3E{\%}3Cleader{\%}3E00000cam
    a2200577 i 4500{\%}3C/leader{\%}3E{\%}3Ccontrolfield
    tag={\%}22001{\%}22{\%}3E016055025{\%}3C/controlfield{\%}3E{\%}3Ccontrolfield
    tag={\%}22003{\%}22{\%}3EMiU{\%}3C/controlfield{\%}3E{\%}3Ccontrolfiel},
  }

  \end{biblist}
  \end{bibdiv}

\end{document}